\documentclass[letterpaper,11pt]{article}

\usepackage[utf8]{inputenc}
\usepackage{amssymb,amsmath,amsthm,mathtools,amsbsy,amsfonts}
\usepackage{subcaption}
\usepackage{graphicx,wasysym}

\usepackage[margin=1.15in]{geometry}
\usepackage{stmaryrd}
\usepackage{cases}
\usepackage{empheq}
\usepackage{paralist}
\usepackage{hyperref}
\usepackage{multimedia}
\usepackage{tikz}
\usepackage{enumitem}
\usepackage{cite}
\usepackage{multirow}
\usepackage{accents}

\usepackage{fullpage}

\usepackage[explicit]{titlesec}
\titleformat{\section}{\large\bfseries\center\raggedright}{\thesection}{0.5em}{{#1}}[]
\titleformat{\subsection}{\bfseries\center\raggedright}{\thesubsection}{0.5em}{{#1}}[]
\titleformat{\subsubsection}[runin]{\bfseries}{\thesubsubsection}{0.5em}{{#1}}[.]
\titlespacing*{\section}{0pt}{0.8\baselineskip}{0.6\baselineskip}
\titlespacing*{\subsection}{0pt}{0.6\baselineskip}{0.4\baselineskip}
\titlespacing*{\subsubsection}{0pt}{0.4\baselineskip}{0.4\baselineskip}

\newcommand\namefont{\normalfont\bfseries}
\newcommand\numberfont{\normalfont\bfseries}
\newcommand\notefont{\normalfont\bfseries}
\newtheoremstyle{mystyle} 
	{0.3em} 
	{0.3em} 
	{\itshape} 
	{} 
	{\normalfont} 
	{.} 
	{.5em} 
	{{\namefont\thmname{#1}}~{\numberfont\thmnumber{#2}}{\notefont\thmnote{ (#3)}}} 
\theoremstyle{plain}
\newtheorem{thm}{Theorem}[section]

\newtheorem{rem}[thm]{Remark}
\newtheorem{lem}[thm]{Lemma}

\newtheorem{prop}[thm]{Proposition}

\newtheorem{defn}[thm]{Definition}

\newtheorem{ex}[thm]{Example}

\allowdisplaybreaks
\mathtoolsset{showonlyrefs}
\renewcommand{\leq}{\leqslant}
\renewcommand{\geq}{\geqslant}

\definecolor{darkred}{rgb}{0.6,0.1,0.1}
\definecolor{darkgreen}{rgb}{0.1,0.6,0.1}
\definecolor{darkblue}{rgb}{0.1,0.1,0.6}

\newcommand{\be}{\begin{equation}}
\newcommand{\ee}{\end{equation}}
\newcommand{\bes}{\begin{equation*}}
\newcommand{\ees}{\end{equation*}}
\newcommand{\bfig}{\begin{figure}}
\newcommand{\efig}{\end{figure}}
\newcommand{\bt}{\begin{table}}
\newcommand{\et}{\end{table}}
\newcommand{\bc}{\begin{center}}
\newcommand{\ec}{\end{center}}

\newcommand{\mt}[1]{\mathrm{#1}}
\def\st{\, \left|\right. \,}
\def\:{\colon}

\newcommand{\abs}[1]{\left\vert#1\right\vert}
\newcommand{\ap}[1]{\left\langle#1\right\rangle}
\newcommand{\norm}[1]{\left\Vert#1\right\Vert}

\def\grad{\nabla}

\DeclareMathOperator{\supp}{supp}

\DeclareMathOperator{\diam}{diam}

\DeclareMathOperator{\Hess}{Hess}

\def\R{\mathbb{R}} 
\def\P{\mathcal{P}} 
\def\M{M} 
\def\Cont{\mt{C}} 

\def\Lip{\mt{Lip}} 
\def\Sph{\mathbb{S}} 

\def\e{\varepsilon}

\def\d{\,\mathrm{d}}

\def\XXint#1#2#3{{\setbox0=\hbox{$#1{#2#3}{\int}$ }
	\vcenter{\hbox{$#2#3$ }}\kern-.6\wd0}}

\def\der{\mathrm{d}}

\def\p{\partial}

\def\bee{\boldsymbol{e}}
\def\bd{\mathcal{W}}

\def\U{\mathcal{U}} 
\def\V{\mathcal{V}} 

\def\barx{\bar{x}}
\def\barz{\bar{z}}
\def\tx{\tilde{x}}
\def\tz{\tilde{z}}

\def\dim{k}
\def\V{v}
\def\PsiX{\Psi^t_X}
\def\PsiY{\Psi^t_Y}
\def\S{{\mathcal{D}_{\varepsilon}}}
\def\Cyl{\mathbb{C}}
\def\C{\mathcal{B}_{\varepsilon}}
\def\North{N}
\def\Lip{\Lambda} 
\def\N{n}
\def\c{C}
\def\h{c}
\def\diam{\Delta}
\def\a{a}
\def\tf{\phi}
\def\pf{\zeta}

\begin{document}

\title{Well-posedness and asymptotic behaviour of an aggregation model with intrinsic interactions on sphere and other manifolds}

\author{Razvan C. Fetecau \thanks{Department of Mathematics, Simon Fraser University, Burnaby, BC V5A 1S6, Canada}
\and Hansol Park \thanks{Department of Mathematical Sciences, Seoul National University, Seoul 08826, Korea}
\and Francesco S. Patacchini \thanks{Department of Mathematical Sciences, Carnegie Mellon University, Pittsburgh, PA 15213, USA} \thanks{IFP Energies nouvelles, 1-4 avenue de Bois-Préau, 92852 Rueil-Malmaison, France}}

\date{\today}

\maketitle


\begin{abstract}
We investigate a model for collective behaviour with intrinsic interactions on Riemannian manifolds. We establish the well-posedness of measure-valued solutions (defined via mass transport) on sphere, as well as investigate the mean-field particle approximation.  We study the long-time behaviour of solutions to the model on sphere, where the primary goal is to establish sufficient conditions for a consensus state to form asymptotically. Well-posedness of solutions and the formation of consensus are also investigated for other manifolds (e.g., a hypercylinder).
\end{abstract}



\section{Introduction}
\label{sect:intro}

We consider a nonlocal aggregation model on a Riemannian manifold $\M$ that consists in the following evolution equation for a population density $\rho$ on $\M$: 
\begin{equation}
\partial_t \rho-\nabla_M \cdot(\rho \nabla_\M K\ast\rho)=0. \label{eqn:model}
\end{equation}
Here, $K\: \M\times \M \to \R$ is an interaction potential, which models social interactions such as attraction and repulsion, and $\nabla_\M \cdot$ and $\nabla_\M $ represent the manifold divergence and gradient, respectively. Also, for a time-dependent measure $\rho_t$ on $\M$, the convolution $K \ast \rho_t$ is given by:
\begin{equation} \label{eqn:conv}
	K * \rho_t(x) = \int_M K(x,y) \d \rho_t(y).
\end{equation}
In \eqref{eqn:model}, we restrict $\rho_t$ to be a probability measure on $\M$ for all $t$, i.e., $\int_M  \d \rho_t=1$ for all $t$.

There has been extensive research on model \eqref{eqn:model} in recent years. The model has many applications, in diverse areas such as swarming in biological groups \cite{M&K}, materials science and granular media \cite{CaMcVi2006}, self-assembly of nanoparticles \cite{HoPu2005}, robotics and space missions \cite{JiEgerstedt2007}, and opinion formation \cite{MotschTadmor2014}. Indeed, the model can capture a wide variety of self-collective or swarm behaviours, such as aggregations on disks, annuli, rings and soccer balls \cite{KoSuUmBe2011,Brecht_etal2011,BrechtUminsky2012}, making it very attractive for applications.  At the same time, model \eqref{eqn:model} in Euclidean space ($\M$ = $\R^\dim$) has been investigated thoroughly by PDE analysis. A partial list of issues addressed in analysis works include the well-posedness of the initial-value problem  \cite{Laurent2007, BertozziLaurent, Figalli_etal2011, BeLaRo2011}, the long-time behaviour of its solutions \cite{LeToBe2009, FeRa10, BertozziCarilloLaurent, FeHuKo11,FeHu13}, and the minimizers for the associated interaction energy \cite{Balague_etalARMA,ChFeTo2015,CaCaPa2015,SiSlTo2015}.

While model  \eqref{eqn:model} in Euclidean spaces has been well studied in literature, there have been far fewer works on the aggregation model posed on arbitrary surfaces or manifolds. In \cite{WuSlepcev2015,CarrilloSlepcevWu2016}, the authors investigate the well-posedness of the aggregation model \eqref{eqn:model} on certain subsets of $\R^\dim$ when interactions depend on the Euclidean distance in the ambient space. Specifically, it is assumed there that the interaction potential $K(x,y)$ is of the form $K(x,y) = K(|x- y|)$, where $|x-y|$ denotes the Euclidean distance in $\R^\dim$ between points $x$ and $y$ on $\M$. A similar assumption is made in various recent works on collective dynamics on matrix manifolds (e.g., orthogonal and unitary groups) \cite{HaKoRy2017,HaKoRy2018}. We will be referring to such models as models with {\em extrinsic} interactions. Emergent behaviours of swarming models and Fokker--Planck-type dynamics with extrinsic interactions on surfaces and manifolds have been investigated (both analytically and numerically) in various papers in recent years; see for instance \cite{LiSpong2014,Li2015,HaKoRy2017,HaKim2019,AlBrCa2020}.
 
In this paper we consider model \eqref{eqn:model} with an interaction potential of the form $K(x,y) = K(d(x,y))$, where $d(x,y)$ is the geodesic distance on $M$ between $x$ and $y$. In other words, we consider model \eqref{eqn:model} with {\em intrinsic} interactions. Such model was proposed and investigated recently in \cite{FeZh2019}, where the authors demonstrate the emergent self-collective behaviour of its solutions on sphere and hyperbolic plane. In particular, it is shown there that solutions can approach asymptotically a diverse set of steady states, that include constant density equilibria, concentrations on geodesic circles, aggregations on geodesic disks and annular regions, and others. Intrinsic interactions are motivated by applications of the model in engineering (robotics) \cite{Gazi:Passino, JiEgerstedt2007}, specifically when individual agents/robots are restricted by environment or mobility constraints to remain on a certain manifold. In such applications, efficient swarming must consider inter-individual  geodesic distances, and hence, be modelled by intrinsic interactions \cite{TronAfsariVidal2012, Markdahl2019}
 
We consider weak, measure-valued solutions to \eqref{eqn:model} defined in the mass transportation sense \cite{CanizoCarrilloRosado2011}. Indeed, equation \eqref{eqn:model} is in the form of a continuity equation, and in geometric terms it represents the transport of the measure $\rho$ along the flow on $\M$ generated by the tangent vector field $\V[\rho]= -\nabla_M K \ast \rho$, which depends on $\rho$ itself \cite{AGS2005}. This general framework includes for instance the case of an interacting particle system and hence it can be used to study particle approximations and mean-field limits. With this interpretation of solutions, a first goal of the present paper is to establish the well-posedness of solutions to model \eqref{eqn:model} set up on a sphere (Section \ref{sect:sphere}) and on a hypercylinder (Section \ref{sect:other}), the main results being stated in Theorems \ref{thm:well-posedness} and \ref{thm:well-posedness-cyl}, respectively. These are the first such results for model \eqref{eqn:model} with intrinsic interactions. We note here that an alternative approach to study well-posedness of solutions is using the theory of gradient flows in the space of probability measures on $\M$ endowed with the Riemannian $2$-Wasserstein metric \cite{AGS2005}. Such techniques were used for nonlinear diffusion equations in \cite{Lisini2009}, and for the interaction equation (with extrinsic interactions) in \cite{WuSlepcev2015}. The approach in our paper, which follows several previous studies of the interaction equation in $\R^\dim$ \cite{CanizoCarrilloRosado2011,CaChHa2014}, is less technical, as it amounts to working with flows of locally Lipschitz vector fields on a manifold. Also, the procedure leads naturally to the mean-field approximation (referred to in the literature as the Dobrushin technique), which can be very useful for numerical simulations.

In working with intrinsic interactions we have to deal with the regularity of the distance function, which is known to be smooth away from cut loci and the diagonal. For this reason we consider interaction potentials that depend on the squared distance $d(x,y)^2$ between points $x$ and $y$ on $\M$ (to avoid  singularities at $x=y$), and also restrict to subsets of the manifold for which no two points are in the cut locus of each other (e.g., an open hemisphere). Note that by this restriction any pair of points on the manifold can be connected by a unique minimizing geodesic. Intuitively, this avoids situations where two interacting particles could be connected by more than one minimizing geodesic and thus would not ``know'' which direction to follow (as for instance, two antipodal points on a sphere).

A second goal of the paper is to investigate the emergence of asymptotic consensus in solutions to model \eqref{eqn:model} with intrinsic interactions on sphere, on hypercylinder and, more generally, on certain product manifolds.
Consensus (also referred to in literature as synchronization or rendezvous) corresponds to an asymptotic state of a delta aggregation in one single point on $M$. Achieving consensus in a network of agents is a very important problem in robotic control  \cite{Sepulchre2011}, in particular when the interactions among agents are intrinsic, as in our paper \cite{TronAfsariVidal2012, Markdahl2019}. Also, such asymptotic states have been of central importance in the Kuramoto oscillator and related models \cite{HaKiPa2015,HaKoRy2017}, as well as in applications of the model to opinion formation \cite{MotschTadmor2014}.  In this paper we prove the formation of consensus equilibria for the intrinsic model on sphere with attractive potentials, as well as asymptotic consensus on certain product manifolds in the specific case of a quadratic interaction potential. To the best of our knowledge, this is the first systematic study of asymptotic behaviour of intrinsic models.
 
The summary of the paper is as follows. In Section \ref{sect:prelims} we present some preliminaries, in particular the notion of solution, some useful results regarding Wasserstein distances, and the main assumption on the interaction potential $K$. Section \ref{sect:sphere} is concerned with the well-posedness of solutions to model \eqref{eqn:model} on sphere (the main result being given in Theorem \ref{thm:well-posedness}), and their stability and mean-field approximation. In Section \ref{sect:synchronization} we investigate the asymptotic behaviour of solutions to model \eqref{eqn:model} on sphere, specifically the formation of consensus equilibria in the continuum and discrete models (Theorems \ref{thm:consensus-cont} and \ref{thm:consensus-disc}). In Section \ref{sect:other} we consider other manifolds (e.g., a hypercylinder) for which we investigate issues such as well-posedness and consensus formation. Finally, Appendix contains some fundamental concepts needed to support the work in the paper, such as general facts on flows on manifolds and how they apply to an interaction velocity field.


\section{Preliminaries and general considerations}
\label{sect:prelims}

Let $\M$ be a smooth, complete and connected $\dim$-dimensional Riemannian manifold, with intrinsic distance $d$. We denote by $\ap{u,v}_x$ and $\norm{u}_x$ the tangent inner product and norm, respectively, for $u,v\in T_x\M$ and $x\in \M$, where $T_x\M$ stands for the tangent space of $\M$ at $x$. The tangent bundle is denoted by $T\M$. We emphasize that throughout this section the manifold $\M$ is not necessarily embedded in $\R^{\dim+1}$.

Unless otherwise mentioned, throughout this paper we use $T\in(0,\infty]$ to denote a generic final time (usually related to existence of solutions) and $\U$ denotes a generic open subset of $\M$.

\subsection{Vector fields and flows on manifolds}
\label{subsect:flows}
Consider a time-dependent vector field $X$ on $\U\times [0,\a)$, for some $\a\in(0,\infty]$, that is, $X \: \U \times [0,\a) \to T\M$ with $X(x,t) \in T_x\M$ for all $(x,t)\in \U\times [0,\a)$. We shall often use the $X_t$ for $X(\cdot,t)$.

Given $\Sigma \subset \U$, a \emph{flow map} generated by $(X,\Sigma)$ is a function $\Psi_X\: \Sigma \times [0,\tau) \to \U$, for some $\tau\leq \a$, that satisfies, for all $x\in\Sigma$ and $t\in[0,\tau)$,
\begin{equation} \label{eq:characteristics-general}
	\begin{cases} \dfrac{\der}{\der t} \Psi^t_X(x) = X_t(\Psi^t_X(x)),\\[10pt]
	\Psi^0_X(x) = x, \end{cases}
\end{equation}
where we used the abbreviation $\Psi^t_X$ for $\Psi_X(\cdot,t)$, which we shall do throughout. Furthermore, a flow map is said to be \emph{maximal} if its time domain cannot be extended while \eqref{eq:characteristics-general} holds; it is said to be \emph{global} if $\tau=\a=\infty$ and \emph{local} otherwise.
In the present paper we are interested in flow maps generated by the velocity field $v[\rho]$ of the interaction equation (see \eqref{eqn:v-field} below), with $\Sigma$ being the support of the initial measure $\rho_0$. In such case we will omit $\Sigma$ and simply say that $v[\rho]$, instead of $(v[\rho],\supp(\rho_0))$, generates a flow map.

The local existence and uniqueness of a flow map follows from standard theory of dynamical systems on manifolds whenever the set $\Sigma$ above is compact; see \cite[Chapter~9]{Lee2013} or \cite[Chapter~4]{AMR1988} for instance. We review some of this theory in Appendix \ref{subsect:A-wp}. In brief, by working in charts and using local coordinates, for a vector field that satisfies a Lipschitz property on charts (see Definition \ref{defn:Lip}) one can make use of standard ODE theory in Euclidean space $\R^\dim$ to establish the local well-posedness of flow maps (Theorem \ref{thm:Cauchy-Lip}). As for $\Sigma$ being compact, this is required to ensure that the maximal time of existence of the flow map is strictly positive.



\subsection{Notion of solution} \label{subsect:solution}

As already mentioned, for the sake of generality, and also because of our future considerations on particle solutions (see Theorem \ref{thm:mfl}), we are interested in defining measure-valued solutions to \eqref{eqn:model}. To this end, denote by $\P(\U)$ the set of Borel probability measures on the metric space $(\U,d)$ and by $\Cont([0,T);\P(\U))$ the set of continuous curves from $[0,T)$ into $\P(\U)$ endowed with the narrow topology. Recall that a sequence $(\rho^n)_{n\geq 1} \subset \P(\U)$ \emph{converges narrowly} to $\rho \in \P(\U)$ if 
\bes
	\int_\U \tf(x) \d\rho^n(x) \to \int_\U \tf(x) \d\rho(x), \qquad \mbox{as $n\to\infty$, for all $\tf \in\Cont_\mt{b}(\U)$,}
\ees
where $\Cont_\mt{b}(\U)$ is the set of continuous and bounded functions on $\U$. 

We denote by $\Psi \# \rho$ the \emph{push-forward} in the mass transportation sense of $\rho$ through a map $\Psi\: \Sigma \to \U$ for some $\Sigma\subset \U$, that is, $\Psi\#\rho$ is the probability measure such that for every measurable function $\zeta\: \U \to [-\infty,\infty]$ with $\zeta\circ \Psi$ integrable with respect to $\rho$, we have
\bes
	\int_\U \zeta(x) \d (\Psi \# \rho)(x) = \int_\Sigma \zeta(\Psi(x)) \d\rho(x).
\ees

Also, for any curve $(\rho_t)_{t\in [0,T)} \subset \P(\U)$, denote by $\V[\rho]\: \U \times [0,T) \to T\M$ the velocity vector field associated to \eqref{eqn:model}, that is,
\begin{equation} \label{eqn:v-field}
	\V[\rho] (x,t) =  -\grad_\M K *\rho_t (x), \qquad \mbox{for all $(x,t) \in \U \times [0,T)$},
\end{equation}
where for convenience we used $\rho_t$ in place of $\rho(t)$, as we shall often do in the following. The convolution in this context is defined as follows: for $h\: \M\times\M \to \R$ and $\rho\in\P(\U)$,
\bes
	h*\rho(x) := \int_\U h(x,y) \d\rho(y).
\ees

Recall the standard notion of solution in the sense of distributions: we say that a curve $(\rho_t)_{t\in [0,T)} \subset \P(\U)$ is a \emph{weak solution in the sense of distributions} to \eqref{eqn:model} if \begin{equation}
	\int_0^T \int_\U \left( \p_t \tf(x,t) + \langle \V[\rho] (x,t), \grad_\M \tf(x,t) \rangle_x \right)  \d\rho_t(x) \d t = 0,  \qquad \mbox{for all $\tf \in \Cont_\mt{c}^\infty(\U \times (0,T))$,}
\end{equation}
where $\Cont_\mt{c}^\infty(\U\times (0,T))$ is the set of smooth and compactly supported functions on $\U\times(0,T)$. For this definition, we implicitly suppose that 
\be\label{eq:bound-distrib}
	\int_S \int_Q \| \V[\rho] (x,t) \|_x \d\rho_t(x) \d t < \infty, \qquad \text{ for all compact sets $S\subset (0,T)$ and $Q \subset \U$}, 
\ee
to ensure, by the Cauchy--Schwarz inequality, that the left-hand side in the definition is finite. 

A solution in the distributional sense can be described in a stronger sense, which is more intuitive and more geometric, as the push-forward of the initial data through the corresponding flow map \cite[Chapter~8.1]{AGS2005}. Indeed, the following result holds: 
\begin{lem}\label{lem:distrib-1}
	Let $(\rho_t)_{t\in[0,T)} \subset \P(\U)$ and suppose that $\V[\rho]$ generates a flow map $\Psi_{\V[\rho]}$ defined on $\supp(\rho_0)\times [0,T)$ and satisfies \eqref{eq:bound-distrib}. Furthermore, assume that $\rho$ satisfies the implicit relation
\be \label{eq:rho-push-forward}
	\rho_t = \Psi^t_{\V[\rho]} \# \rho_0, \qquad \mbox{for all $t\in[0,T)$}.
\ee
Then, $\rho$ belongs to $\Cont([0,T);\P(\U))$ and is a weak solution in the sense of distributions to equation \eqref{eqn:model}.
\end{lem}
The proof follows closely \cite[Lemma 8.1.6]{AGS2005}. For completeness, we provide it in Appendix \ref{subsect:A-sol}. In other words, it suffices to find a curve of the form \eqref{eq:rho-push-forward} satisfying \eqref{eq:bound-distrib} to show existence of a solution in the sense of distributions to the interaction equation. This motivates the following definition of weak, or measure, solution (see also \cite{CanizoCarrilloRosado2011}):
\begin{defn}[Notion of solution]
	We say that $(\rho_t)_{t\in[0,T)} \subset \P(\U)$ is a \emph{weak solution} to \eqref{eqn:model} if $\V[\rho]$ generates a unique flow map $\Psi_{\V[\rho]}$ defined on $\supp(\rho_0)\times [0,T)$ and \eqref{eq:rho-push-forward} holds.
\end{defn}

From the proof of Lemma \ref{lem:distrib-1} in Appendix \ref{subsect:A-sol}, we see that any weak solution belongs to $\Cont([0,T);\P(\U))$, whether or not it satisfies \eqref{eq:bound-distrib}.


\subsection{Wasserstein distance}
\label{subsect:wasserstein}

To compare solutions to \eqref{eqn:model} we will use the intrinsic $1$-Wasserstein distance: for all $\rho,\sigma \in \P(\U)$,
\bes
	W_1(\rho,\sigma) = \inf_{\pi \in \Pi(\rho,\sigma)} \int_{\U\times\U} d(x,y) \d\pi(x,y),
\ees
where $\Pi(\rho,\sigma) \subset \P(\U\times\U)$ is the set of transport plans between $\rho$ and $\sigma$, i.e., the set of elements in $\P(\U\times\U)$ with first and second marginals $\rho$ and $\sigma$, respectively. 

We write $\P_1(\U)$ the set of probability measures on $\U$ with finite first moment and $\P_\infty(\U) \subset \P_1(\U)$ the set of probability measures on $\U$ with compact support; we have that $(\P_1(\U),W_1)$ (and thus $(\P_\infty(\U),W_1)$) is a well-defined metric space. We furthermore metrize the space $\Cont([0,T);\P_1(\U))$ (and thus $\Cont([0,T);\P_\infty(\U))$) with the distance defined by
\bes
	\bd_1(\rho,\sigma) = \sup_{t \in [0,T)} W_1(\rho_t,\sigma_t), \qquad \mbox{for all $\rho,\sigma \in \Cont([0,T);\P_1(\U))$}.
\ees

We give a preliminary lemma first, analogous to results in \cite[Lemmas 3.11--3.13]{CanizoCarrilloRosado2011}, which considers various Lipschitz properties of $W_1$.
\begin{lem}\label{lem:preliminary}
	The following four statements hold.
	\begin{enumerate}[label=(\roman*)]
		\item\label{it:prel1} Let $\Sigma\subset \U$. Let furthermore $\rho \in \P_1(\U)$ with $\supp(\rho) \subset \Sigma$ and $\Psi_1,\Psi_2\:\Sigma \to \U$ be measurable functions. Then,
			\bes
				W_1({\Psi_1}\#\rho,{\Psi_2}\#\rho) \leq \sup_{x \in \supp(\rho)} d(\Psi_1(x),\Psi_2(x)).
			\ees
		\item\label{it:prel2} Let $\a\in(0,\infty]$ and let $X$ be a time-dependent vector field on $\U\times[0,\a)$. Let $\rho \in \P_1(\U)$ and suppose that $(X,\supp(\rho))$ generates a flow map $\Psi_X$ defined on $\supp(\rho)\times[0,\tau)$ for some $\tau\leq \a$. Suppose furthermore that $X$ is bounded on $\U\times [0,\tau)$, i.e., there exists $C>0$ such that $\norm{X(x,t)}_{x\in \U}<C$ for all $x\in \U$ and $t\in[0,\tau)$. Then,
			\bes
				W_1({\Psi^t_X} \# \rho,{\Psi^s_X} \# \rho) \leq C|t-s|, \quad \quad \mbox{for all $t,s \in [0,\tau)$}.
			\ees
		\item\label{it:prel3} Let $\Sigma\subset\U$ and let $\Psi\: \Sigma \to \U$ be Lipschitz continuous as a map from the metric space $(\Sigma,d)$ into the metric space $(\U,d)$; denote by $L_\Psi$ its Lipschitz constant. Moreover, let $\rho,\sigma \in \P_\infty(\U)$. Then,
			\bes
				W_1(\Psi\#\rho,\Psi\#\sigma) \leq L_\Psi W_1(\rho,\sigma).
			\ees
	\end{enumerate}
\end{lem}
\begin{proof}
	Let us first show \ref{it:prel1}. Consider the transport plan given by $\pi = (\Psi_1, \Psi_2)\#\rho$, where we define $(\Psi_1,\Psi_2)\: \Sigma \to \U\times\U$ by
\bes
	(\Psi_1,\Psi_2)(x) = (\Psi_1(x),\Psi_2(x)), \qquad \mbox{for all $x \in \Sigma$.}
\ees
Then, $\pi$ has ${\Psi_1}\#\rho$ and ${\Psi_2}\#\rho$ as first and second marginals, respectively, and therefore $\pi \in \Pi({\Psi_1}\#\rho,{\Psi_2}\#\rho)$. We get
\begin{align*}
	W_1({\Psi_1}\#\rho,{\Psi_2}\#\rho) &\leq \int_{\U\times\U} d(x,y) \d \pi(x,y) = \int_{\supp(\rho)} d(\Psi_1(x),\Psi_2(x)) \d\rho(x)\\
	&\leq \sup_{x\in \supp(\rho)} d(\Psi_1(x),\Psi_2(x)),
\end{align*}
where we used that $\rho$ is a probability measure on $\U$.

Let us now prove \ref{it:prel2}. Let $t,s\in[0,\tau)$. We have, from \ref{it:prel1}, 
\begin{equation}\label{eq:W1-flow}
	W_1({\Psi^t_X}\# \rho,{\Psi^s_X}\# \rho) \leq \sup_{x \in \supp(\rho)} d(\Psi^t_X(x),\Psi^s_X(x)).
\end{equation}
Without loss of generality, assume $t>s$ and with $x$ fixed in $\supp(\rho)$ consider the curve $b \mapsto \Psi^b_X(x)$ on $\M$, with $s \leq b \leq t$. The length $\mathcal{L}$ of this curve, that joins $\Psi^s_X(x)$ and $\Psi^t_X(x)$, can be bounded above using \eqref{eq:characteristics-general} as: 
\be\label{eq:length-vector-field}
	\mathcal{L} = \int_s^t \| X_b(\Psi^b_X(x))\|_{\Psi^b_X(x)} \d b \leq C |t-s|.
\ee
The conclusion now comes from \eqref{eq:length-vector-field}, \eqref{eq:W1-flow}, and the fact that $d(\Psi^t_X(x),\Psi^s_X(x)) \leq \mathcal{L}$.

Finally, let us prove \ref{it:prel3}. Let $\pi$ be an optimal transport plan between $\rho$ and $\sigma$, so that $\supp(\pi) \subset \U\times \U$. Then the plan $\bar\pi = (\Psi,\Psi)\#\pi$ has $\Psi\#\rho$ and $\Psi\#\sigma$ as first and second marginals, respectively, so that $\bar\pi \in \Pi(\Psi\#\rho,\Psi\#\sigma)$. Thus,
\begin{align*}
	W_1(\Psi\#\rho,\Psi\#\sigma) &\leq \int_{\U\times\U} d(x,y) \d \bar\pi(x,y)\\
	&= \int_{\Sigma\times\Sigma} d(\Psi(x),\Psi(y)) \d \pi(x,y)\\ 
	&\leq L_\Psi \int_{\Sigma\times \Sigma} d(x,y) \d\pi(x,y) \leq L_\Psi W_1(\rho,\sigma). \qedhere
\end{align*}
\end{proof}

Note that Lemma \ref{lem:preliminary}\ref{it:prel2} shows that if the velocity field $v[\sigma]$ is bounded for any curve $\sigma$ on $[0,T)$ of probability measures, then any weak solution $\rho$ on $[0,T)$ to the interaction equation starting from an element in $\P_1(\U)$ is in fact Lipschitz continuous in time, and in particular absolutely continuous in time. Indeed, in this case,
\bes
	W_1(\rho_t,\rho_s) = W_1(\Psi_{\V[\rho]}^t\#\rho_0,\Psi_{\V[\rho]}^s\#\rho_0) \leq C |t-s| \quad \mbox{for all $t,s \in [0,T)$},
\ees
where $\Psi_{\V[\rho]}$ is the unique flow map generated by $\V[\rho]$ on the time interval $[0,T)$ and $C$ is the constant from Lemma \ref{lem:preliminary}\ref{it:prel2}.


\subsection{Assumption on the interaction potential}
\label{subsect:potential}

We assume that $K\:\M\times\M \to \R$ depends only on the intrinsic distance $d$ on $\M$. To avoid issues regarding the differentiability of the distance function on the diagonal $\{(x,y) \in \M\times \M \st x = y\}$, we take in fact $K$ to depend on the squared distance function instead. Specifically, we make the following assumption on the interaction potential:

\begin{enumerate}[label=\textbf{(H)}]
\item \label{hyp:K} $K\: \M \times \M \to \R$ has the form
	\begin{equation}
	\label{eqn:K-gen}
		K(x,y) = g(d(x,y)^2), \qquad \mbox{for all } x,y\in \M,
	\end{equation}
	where $g\: [0,\infty) \to \R$ is differentiable, with locally Lipschitz continuous derivative. 
\end{enumerate}

In the following we use the notation $K_y(x)$ for $K(x,y)$ and $d_y(x)$ for $d(x,y)$. Given the expression \eqref{eqn:K-gen} of $K$, its gradient can be computed as
\begin{equation}
\label{eqn:gradK-gen}
	\nabla_\M K_y(x) = -2 g'(d(x,y)^2) \log_x y, 
\end{equation}
where we used the chain rule and the fact that 
\begin{equation}
\label{eqn:gradd}
	\nabla_{M} d_y(x) = -\frac{\log_x y}{d(x,y)}, \qquad \mbox{for $x\neq y$}.
\end{equation}
Here, $\log_x y$ denotes the Riemannian logarithm map (i.e., the inverse of the Riemannian exponential map) on $\M$ \cite{Petersen2006}. Equations \eqref{eqn:gradK-gen} and \eqref{eqn:gradd} only hold for points $y$ within the injectivity radius of $M$ at $x$ (or, equivalently, away from the cut locus of $x$). To ensure that these formulas hold, we shall therefore restrict in the following to an open subset $\U$ of $\M$ which is geodesically convex; we remark that, in particular, this implies that $\U$ can be covered by a single chart.

We also note here that the physical interpretation of \eqref{eqn:model} as an aggregation model is encoded in \eqref{eqn:v-field} and \eqref{eqn:gradK-gen}. Specifically, by interacting with a point mass at location $y$, the point mass at $x$ is driven by a force of magnitude proportional to $|g'(d(x,y)^2|d(x,y)$, to move either towards $y$ (provided $g'(d(x,y)^2) >0$) or away from $y$  (provided $g'(d(x,y)^2) < 0$). The velocity field at $x$ computed by \eqref{eqn:v-field} takes into account all contributions from interactions with point masses $y \in M$ through the nonlocality induced by the convolution. 


\section{Intrinsic aggregation model on the unit sphere}
 \label{sect:sphere}
In this section we take the Riemannian manifold $\M$ to be the $\dim$-dimensional unit sphere $\Sph^{\dim}$ and show the well-posedness of model \eqref{eqn:model} in the case when the dynamics is restricted to a geodesically convex subset of an open hemisphere. Note that here, by compactness, $\P_1(\Sph^{\dim}) = \P_\infty(\Sph^\dim) = \P(\Sph^{\dim})$.

We equip $\Sph^\dim$ with the induced metric from $\R^{\dim+1}$; in particular, this means that we shall equivalently regard points in $\Sph^\dim$ and tangent vectors of $\Sph^\dim$ as vectors in $\R^{\dim+1}$, with the property that $\ap{u,v}_x=u\cdot v$ for all $u,v\in T_x\Sph^\dim$ and $x\in\Sph^\dim$, where $u\cdot v$ stands for the canonical inner product in $\R^{\dim+1}$ of $u$ and $v$.


\subsection{Intrinsic distance}
\label{subsect:distanceS}

Given $x,y \in \Sph^\dim$, the Riemannian, or intrinsic, distance between points $x,y \in \Sph^{\dim}$ is given by:
\bes
	d(x,y) = \theta_{xy} \in [0,\pi], 
\ees
where $\theta_{xy} = \arccos(x \cdot y)$ represents the angle made by the vectors $x$ and $y$. Based on the observation above, to have a well-defined gradient of the distance function, we consider a subset of the sphere where no two points are in the cut locus of each other. Specifically, fix any $\e\in (0,\pi/2)$ and without any loss of generality choose the following open and geodesically convex subset:
\begin{equation}
\label{eqn:setS}
	\S = \left\{ x \in \Sph^{\dim} \st d(x,\North) < \frac\pi2 - \e \right\},
\end{equation}
where $\North = (0,\dots,0,1)$ represents the North pole of the unit sphere. Note that the maximum distance on $\S$ is bounded by $\pi - 2 \e<\pi$. 

On $\S$, which here plays the role of $\U$ in the general setting of the previous section, the logarithm map is given explicitly by
\begin{equation}
\label{eqn:log-sphere}
	\log_x y = \frac{\theta_{xy}}{\sin(\theta_{xy})} (y - \cos (\theta_{xy}) x), \qquad \mbox{for all $x,y\in \S$},
\end{equation}
and by \eqref{eqn:gradd},
\begin{equation}\label{eqn:gradd-sphere}
	\grad_{\Sph^{\dim}} d_y(x)  = \frac{\cos(\theta_{xy}) x -y}{\sin (\theta_{xy})}, \qquad \mbox{for all $x,y\in \S$ with $x\neq y$}.
\end{equation}
As $d$ is a distance function, one can check indeed that $|\grad_{\Sph^{\dim}} d_y(x) | =1$ for all $x,y\in\S$ with $x\neq y$.

For convenience of notation, set $f(\theta) := \theta/\sin(\theta)$ for $\theta\in[0,\pi)$, and hence, for all $x,y\in\S$,
\begin{equation}
	\grad_{\Sph^{\dim}} d^2_y(x)  = 2 f(\theta_{xy})(\cos(\theta_{xy}) x -y).
	\label{eqn:gradd2-sphere}
\end{equation}
Note that $f(\theta) \to \infty$ as $\theta\to\pi$, which illustrates quantitatively why we need to restrict to a geodesically convex subset of $\Sph^{\dim}$: this prevents the gradient of the squared distance from blowing up by not allowing any points $x$ and $y$ to be in the cut locus of each other (i.e., from being antipodal and have $\theta_{xy}=\pi$).

As $f$ and $f'$ are bounded on $[0,\pi - 2 \e]$, denote:
\bes
	C_f(\e) := \sup_{\theta \in [0,\pi - 2 \e]} f(\theta), \qquad L_f(\e) := \sup_{\theta \in [0,\pi - 2 \e]} f'(\theta).
\ees
Both $C_f(\e)$ and $L_f(\e)$ blow up as $\e \to 0$, which justifies the choice of $\e>0$ in the definition of $\S$. Also, since by Assumption \ref{hyp:K} the function $g'$ is locally Lipschitz continuous, denote by $C_{g'}(\e)$ and $L_{g'}(\e)$ the $L^\infty$ norm and the Lipschitz constants of $g'$ on $[0,(\pi - 2 \e)^2]$, respectively.


\subsection{Vector fields on $\S$}
\label{subsect:vectorsS}

Our approach in what follows relies on the fact that $\Sph^\dim$ is embedded in $\R^{\dim+1}$, which allows us to view vector fields in $\S$ as vector fields in $\R^{\dim+1}$ and in particular, to take the difference of tangent vectors at different points of $\S$. We give here two important lemmas for flows of Lipschitz vector fields on $\S$. We will require that the vector fields satisfy a Lipschitz condition (see \eqref{eq:Lipschitz-X}) with respect to the norm of the ambient space $\R^{\dim+1}$, denoted by $\abs{\cdot}$. 
As shown later in Lemma \ref{lem:dist-flow-maps-K}, the vector field associated to the interaction equation satisfies indeed this Lipschitz property.

\begin{lem}\label{lem:dist-flow-maps}
	Let $X,Y$ be two time-dependent vector fields on $\S$. Let $\Sigma\subset \S$ and suppose that $\Psi_X$ and $\Psi_Y$ are flow maps defined on $\Sigma\times [0,\tau)$, for some $\tau >0$, generated by $(X,\Sigma)$ and $(Y,\Sigma)$, respectively. Assume furthermore that $X$ is bounded on $\S\times[0,\tau)$ and Lipschitz continuous with respect to its first variable (uniformly with respect to $t \in [0,\tau)$) on $\S\times [0,\tau)$, i.e., there exists $L_X>0$ such that
	\be 
		|X(x,t) - X(y,t)| \leq L_X\, d(x,y), \quad \quad \mbox{for all $(x,y,t) \in \S \times \S \times [0,\tau)$}.
	\label{eq:Lipschitz-X}	
	\ee
	Then, for all $p\in\Sigma$,
	\bes
		d(\PsiX(p),\PsiY(p)) \leq \frac{e^{(L_X+2\|X\|_{L^\infty(\S\times [0,\tau))})t} -1}{L_X+2\|X\|_{L^\infty(\S\times [0,\tau))}} \|X-Y\|_{L^\infty(\S\times [0,\tau))},\quad \quad \mbox{for all $t\in[0,\tau)$}.
	\ees
\end{lem}
\begin{proof}

Fix $p\in\Sigma$ and $t\in[0,\tau)$. We have to estimate the distance $d(\PsiX(p),\PsiY(p))$ when $\PsiX(p)\neq \PsiY(p)$ (otherwise the result is trivial). Compute
\begin{align} \label{eqn:est0}
	\frac{\der}{\der t} d(\PsiX(p),\PsiY(p)) & = \nabla_{\Sph^\dim}\, d_{\PsiY(p)}(\PsiX(p)) \cdot X_t(\PsiX(p)) + \nabla_{\Sph^\dim}\, d_{\PsiX(p)}(\PsiY(p)) \cdot Y_t(\PsiY(p)) \nonumber \\
	& := I + II.
\end{align}
Add and subtract 
\[
A: = \nabla_{\Sph^{\dim}}\, d_{\PsiX(p)}(\PsiY(p)) \cdot X_t(\PsiX(p))
\]
and
\[
B:=  \nabla_{\Sph^{\dim}}\, d_{\PsiX(p)}(\PsiY(p)) \cdot X_t(\PsiY(p))
\] 
to the right-hand side of \eqref{eqn:est0}, which now reads:
\[
I+ A -A + B - B + II. 
\]
The terms $II - B$ estimate as:
\begin{align}
II - B & =  \nabla_{\Sph^\dim}\, d_{\PsiX(p)}(\PsiY(p)) \cdot (Y_t(\PsiY(p)) - X_t(\PsiY(p))) \nonumber \\[5pt]
& \leq \| X-Y\|_{L^\infty([0,\tau)\times \S)},
\label{eqn:est1}
\end{align}
where we used the Cauchy--Schwarz inequality and the fact that the gradient of the distance has norm equal to $1$. By similar considerations, also estimate:
\begin{align}
B - A &=   \nabla_{\Sph^\dim}\, d_{\PsiX(p)}(\PsiY(p)) \cdot (X_t(\PsiY(p)) - X_t(\PsiX(p))) \nonumber \\[5pt]
& \leq |X_t(\PsiY(p)) - X_t(\PsiX(p))|  \nonumber \\[5pt]
& \leq L_X \, d(\PsiX(p),\PsiY(p)),
 \label{eqn:est2}
\end{align}
where for the last inequality we used  the Lipschitz condition \eqref{eq:Lipschitz-X}.

Finally, writing $\theta = d(\PsiX(p),\PsiY(p))$ and using \eqref{eqn:gradd-sphere},
\begin{align}
I + A &=  \left( \nabla_{\Sph^\dim}\, d_{\PsiY(p)}(\PsiX(p)) + \nabla_{\Sph^\dim}\, d_{\PsiX(p)}(\PsiY(p)) \right) \cdot  X_t(\PsiX(p)) \nonumber \\[5pt]
& = \frac{1}{\sin \theta} \left( \cos \theta \, \PsiX(p) - \PsiY(p) + \cos \theta \, \PsiY(p) - \PsiX(p) \right) \cdot X_t(\PsiX(p)).
\label{eqn:t3}
\end{align}
To estimate \eqref{eqn:t3}, write $x = \PsiX(p)$ and $y = \PsiY(p)$; note that $\theta$ is the angle between $Ox$ and $Oy$, where $O$ is the centre of the sphere; see Figure \ref{fig:figure} for an illustration. The vector $\cos \theta \, x - y$ at $x$ is tangent to the great circle containing $x$ and $y$, pointing ``away" from $y$. Analogous comment for the vector $\cos \theta \, y - x$ at $y$. Consider two orthogonal directions in the plane $Oxy$: one along the bisector of the angle $xOy$ and the other perpendicular to it; these directions are indicated by dotted lines in Figure \ref{fig:figure}. Also note that $\theta <\pi$. By symmetry, $\cos \theta \, x - y$ and $\cos \theta \, y - x$ have the same component in the bisector direction, and opposite components in the orthogonal direction. The latter two cancel each other when the two vectors are added. The components in the bisector direction combine. 

\begin{figure}[htb]
 \begin{center}
  \includegraphics[width=0.35\textwidth]{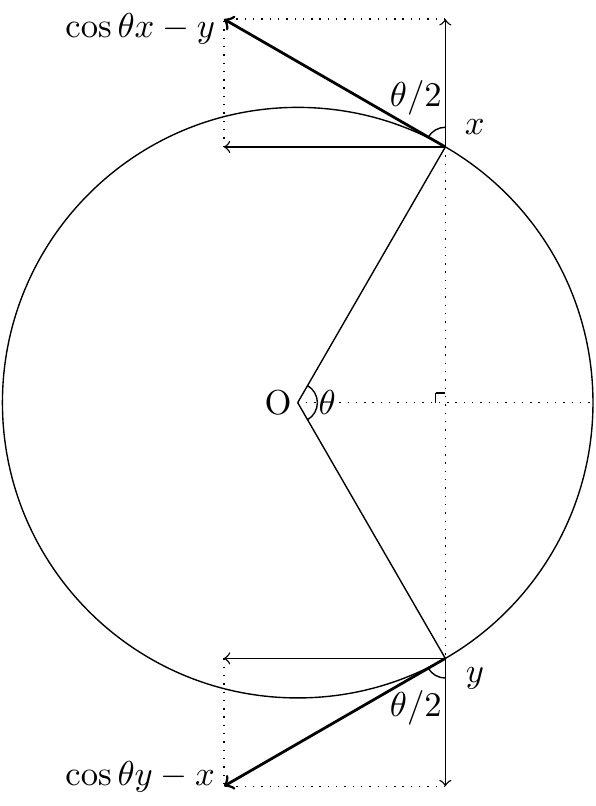} 
\caption{A great circle on the unit sphere containing points $x$ and $y$. The vectors $\cos \theta \, x - y$ at $x$ and 
$\cos \theta \, y - x$  at $y$ have the same component in the direction of the bisector of the angle $xOy$, and opposite components in the direction orthogonal to the bisector.}
\label{fig:figure}
\end{center}
\end{figure}

By the geometry of the problem, $\cos \theta \, x - y$ makes an angle of $\frac{\pi}{2}-\frac{\theta}{2}$ with the bisector direction. Given that $|\cos \theta \, x - y| = 
|\cos \theta \, y - x| = \sin \theta$, and the considerations above,
\begin{equation}
|\cos \theta \, x - y + \cos \theta \, y - x| =
2 \sin \theta \cos \left(\frac{\pi}{2} - \frac{\theta}{2} \right) = 2 \sin \theta \sin \left( \frac{\theta}{2}\right).
\label{eqn:t3-int}
\end{equation}
Now we return to \eqref{eqn:t3} and estimate using the Cauchy--Schwarz inequality and \eqref{eqn:t3-int}: 
\begin{align}
\frac{1}{\sin \theta} \left( \cos \theta \, \PsiX(p) - \PsiY(p) + \cos \theta \, \PsiY(p) - \PsiX(p) \right) \cdot X_t(\PsiX(p)) & \leq 2 \|X\|_{L^\infty(\S\times[0,\tau)} \sin \left( \frac{\theta}{2}\right) \nonumber \\
& \leq 2 \|X\|_{L^\infty(\S\times[0,\tau))} \theta,
\label{eqn:est3}
\end{align}
where we used for the second inequality that $\sin \left( \frac{\theta}{2}\right) \leq \theta$.

Collecting \eqref{eqn:est1}, \eqref{eqn:est2} and \eqref{eqn:est3} we find from \eqref{eqn:est0}:
\begin{equation}
\label{eqn:est-total}
	\frac{\der}{\der t} d(\PsiX(p),\PsiY(p)) \leq (L_X + 2 \|X\|_{L^\infty(\S\times[0,\tau))}) \, d(\PsiX(p),\PsiY(p)) +  \| X-Y\|_{L^\infty(\S\times[0,\tau))},
\end{equation}
and Gronwall's lemma gives the desired estimate.
\end{proof}

\begin{lem}\label{lem:Lipschitz-initial}
Let $X$ be a time-dependent vector field on $\S$. Let $\Sigma\subset \S$ and suppose that $\Psi_X$ is a flow map defined on $\Sigma\times [0,\tau)$, for some $\tau >0$, generated by $(X,\Sigma)$. Assume moreover that $X$ is bounded on $\S\times [0,\tau)$ and Lipschitz continuous with respect to its first variable on $\S\times [0,\tau)$ (i.e., it satisfies \eqref{eq:Lipschitz-X}) with Lipschitz constant $L_X>0$. Then, 
\bes
	d(\PsiX(p),\PsiX(q)) \leq e^{(L_X+2\|X\|_{L^\infty(\S\times[0,\tau))})t} d(p,q), \qquad \mbox{for all $p,q \in \Sigma$ and $t\in [0,\tau)$}.
\ees
\end{lem}
\begin{proof}
Fix $t\in[0,\tau)$ and $p,q \in \Sigma$ and estimate the distance $d(\PsiX(p),\PsiX(q))$ by computing:
\begin{equation} \label{eqn:estpq0}
	\frac{\der}{\der t} d(\PsiX(p),\PsiX(q)) = \nabla_{\Sph^\dim}\, d_{\PsiX(q)}(\PsiX(p)) \cdot X_t(\PsiX(p)) + \nabla_{\Sph^\dim}\, d_{\PsiX(p)}(\PsiX(q)) \cdot X_t(\PsiX(q)). 
\end{equation}
Add and subtract $\nabla_{\Sph^\dim}\, d_{\PsiX(p)}(\PsiX(q)) \cdot X_t(\PsiX(p))$ to the right-hand side above. By considerations similar to those used in the proof of Lemma \ref{lem:dist-flow-maps} (see the estimates on the term $I+A$ leading to \eqref{eqn:est3}), one gets:
\[
	\left( \nabla_{\Sph^\dim}\, d_{\PsiX(q)}(\PsiX(p)) + \nabla_{\Sph^\dim}\, d_{\PsiX(p)}(\PsiX(q)) \right) \cdot X_t(\PsiX(p)) \leq 2 \|X\|_{L^\infty(\S\times[0,\tau))} d(\PsiX(p),\PsiX(q)).
\]
Also, by the Cauchy--Schwarz inequality and the Lipschitz condition on $X$, 
\[
\nabla_{\Sph^\dim}\, d_{\PsiX(p)}(\PsiX(q)) \cdot (X_t(\PsiX(q)) - X_t(\PsiX(p))) \leq L_X d(\PsiX(p),\PsiX(q)).
\]
Using the two estimates above in \eqref{eqn:estpq0} one then finds:
\[
\frac{\der}{\der t} d(\PsiX(p),\PsiX(q)) \leq ( L_X + 2 \|X\|_{L^\infty(\S\times[0,\tau))}) d(\PsiX(p),\PsiX(q)),
\]
which by Gronwall's lemma yields the desired result.
\end{proof}


\subsection{Well-posedness of solutions}
\label{subsect:well-posedness}

We first show that on sphere, the vector field \eqref{eqn:v-field} associated to equation \eqref{eqn:model} is bounded and satisfies the Lipschitz condition needed to apply the Lemmas in Section \ref{subsect:vectorsS}.

\begin{lem}\label{lem:dist-flow-maps-K}
 	Let $K$ satisfy \ref{hyp:K} and let $\rho \in \Cont([0,T);\P(\S))$. Then, the vector field $\V[\rho]$ given by \eqref{eqn:v-field} is bounded on $\S\times [0,T)$ and satisfies the Lipschitz condition \eqref{eq:Lipschitz-X}, that is, there exists $L(\e)>0$ such that
\bes
		|\V[\rho](x,t) - \V[\rho](y,t)| \leq L(\e) d(x,y), \qquad \mbox{for all $(x,y,t) \in \S \times \S\times [0,T)$}.
\ees
More specifically,
\bes
	\|\V[\rho]\|_{L^\infty(\S\times[0,T))} \leq 2 \pi C_{g'}(\e),
\ees
and the Lipschitz constant $L(\e)$ depends only on $C_f(\e),L_f(\e),C_{g'}(\e)$ and $L_{g'}(\e)$.
\end{lem}
\begin{proof}
The boundedness of $\V[\rho]$ is immediate. Indeed, for all $(x,t) \in \S\times [0,T)$,
\be\label{eqn:X-bound}
	|\V[\rho](x,t)| \leq \int_{\S} |\grad_{\Sph^\dim} K_y(x)| \d\rho_t(y) = \int_{\S} |g'(d(x,y)^2) \grad_{\Sph^\dim} d_y^2(x)| \d\rho_t(y) \leq 2 \pi C_{g'}(\e), 
\ee
where for the last inequality we used the relation $|\grad_{\Sph^\dim} d_y^2(x)| = 2d(x,y)$, the bound on $g'$ and \eqref{eqn:gradd2-sphere} (note that $d(x,y)<\pi$ for every $x,y \in \S$).

For the Lipschitz condition, let $x,y \in \S$. By \eqref{eqn:v-field} we have
\begin{equation}
\label{eqn:X-diff}
\V[\rho](x,t) - \V[\rho](y,t) = \int_{\S} (\grad_{\Sph^\dim} K_z(x) - \grad_{\Sph^\dim} K_z(y)) \d \rho_t(z).
\end{equation}
As noted before, taking the difference of tangent vectors at different points $x$ and $y$ makes sense here since, in the case of the sphere, tangent vectors can be regarded as vectors in the embedding space $\R^{\dim+1}$. Compute, for all $z\in\S$, using \eqref{eqn:gradd2-sphere}:
\[
	\grad_{\Sph^\dim} d_z^2(x) - \grad_{\Sph^\dim} d_z^2(y) = 2 f(\theta_{xz}) \left( \cos (\theta_{xz}) \, x - z \right) -  2 f(\theta_{yz}) \left( \cos (\theta_{yz}) \, y - z \right),
\]
where $\theta_{xz} := d(x,z)$ and $\theta_{yz} := d(y,z)$.

Add and subtract $2 f(\theta_{xz}) \cos(\theta_{xz}) y + 2 f(\theta_{yz}) \cos(\theta_{xz})y$ in the right-hand side above to get:
\begin{align*}
	& \grad_{\Sph^\dim} d_z^2(x) - \grad_{\Sph^\dim} d_z^2(y)\\ 
	&\qquad = 2 f(\theta_{xz}) \cos(\theta_{xz})(x-y) + 2 (f(\theta_{xz}) - f(\theta_{yz}))\cos(\theta_{xz})y + 2 f(\theta_{yz})(\cos(\theta_{xz}) - \cos(\theta_{yz}))y\\
	&\qquad  -2 \left( f(\theta_{xz}) - f(\theta_{yz}) \right) z.
\end{align*}
This yields, for all $z\in\S$,
\begin{align}
	& |\grad_{\Sph^\dim} d_z^2(x) - \grad_{\Sph^\dim} d_z^2(y)| \nonumber  \\
	&\qquad \leq 2 C_f(\e) |x-y| + 2 L_f(\e)|\theta_{xz}-\theta_{yz}| + 2 C_f(\e)|\theta_{xz}-\theta_{yz}| + 2 L_f(\e) |\theta_{xz}-\theta_{yz}| \nonumber  \\
	&\qquad \leq 4(C_f(\e) + L_f(\e)) d(x,y), \label{eqn:est-d2}
\end{align}
where in the last inequality we used the triangle inequality $|\theta_{xz}-\theta_{yz}| = |d(x,z)-d(y,z)| \leq d(x,y)$, and that the Euclidean distance in $\R^{\dim+1}$ is less than or equal to the induced distance on $\Sph^\dim$. 

For an interaction potential in the form \eqref{eqn:K-gen}, one then gets, for all $z\in\S$:
\begin{align*}
	& |\grad_{\Sph^\dim} K_z(x) - \grad_{\Sph^\dim} K_z(y)| = |g'(d(x,z)^2) \grad_{\Sph^\dim} d_z^2(x) - g'(d(y,z)^2) \grad_{\Sph^\dim} d_z^2(y)|\\
	&\qquad \qquad \leq |g'(\theta_{xz}^2) - g'(\theta_{yz}^2)| |\grad_{\Sph^\dim}d_z^2(x)| + |g'(\theta_{yz}^2)| |\grad_{\Sph^\dim} d_z^2(x) - \grad_{\Sph^\dim} d_z^2(y)|\\
	&\qquad \qquad \leq 2 L_{g'}(\e) |\theta_{xz}+\theta_{yz}||\theta_{xz}-\theta_{yz}| \theta_{xz} + 4C_{g'}(\e)(C_f(\e) + L_f(\e)) d(x,y)\\
	&\qquad \qquad \leq (4 \pi^2 L_{g'}(\e)+ 4 C_{g'}(\e)(C_f(\e) + L_f(\e))) d(x,y).
\end{align*}
For the first inequality above we added and subtracted $g'(\theta_{yz}^2) \grad_{\Sph^\dim}d_z^2(x)$ on the first line and then used triangle inequality. For the second inequality we used \eqref{eqn:est-d2}, the bounds and Lipschitz constants of $g'$, and the fact that $|\grad_{\Sph^\dim}d_z^2(x)| = 2 \theta_{xz}$. Finally, for the last inequality we used $|\theta_{xz} - \theta_{yz}| \leq d(x,y)$ by triangle inequality, and that $\theta_{xz},\theta_{yz} < \pi$.

Set 
\[
L(\e) := 4 \pi^2 L_{g'}(\e)  + 4C_{g'}(\e)(C_f(\e) + L_f(\e)).
\] 
Then, for all $t\in [0,T)$, by \eqref{eqn:X-diff} and the estimate above we get:
\bes
\label{eq:Lip-flow}
	|\V[\rho](x,t) - \V[\rho](t,y)| \leq L(\e) d(x,y) \int_{\S} \d\rho_t(z) = L(\e) d(x,y),
\ees
where we also used that $\rho_t$ is a probability measure on $\S$.
\end{proof}

\begin{rem}
In Lemma \ref{lem:dist-flow-maps-K}, the upper bound on $\|\V[\rho]\|_{L^\infty(\S\times[0,T))}$ and the Lipschitz constant of $\V[\rho]$ do not depend on the curve $\rho$. This is important for subsequent considerations, in particular for the proof of Theorem \ref{thm:well-posedness}, the main result in this section. The lemma also ensures that the implicit condition \eqref{eq:bound-distrib} holds.
\end{rem}

The following lemma is a fundamental step towards the proof of well-posedness; see for instance \cite[Lemma~3.15]{CanizoCarrilloRosado2011}, and also \cite[Theorem~4.1]{CaChHa2014}.

\begin{lem}\label{lem:grad-lip-2}
	Let $K$ satisfy \ref{hyp:K} and let $\rho,\sigma \in \Cont([0,T);\P(\S))$. Then,
	\begin{equation}
		\|\V[\rho]-\V[\sigma]\|_{L^\infty([0,T)\times \S)} \leq \Lip(\e) \bd_1(\rho,\sigma),
	 \label{eqn:X-Yest}
	\end{equation}
	where $\Lip(\e)$ is a constant depending on $C_f(\e)$, $L_f(\e)$, $C_{g'}(\e)$, and $L_{g'}(\e)$.
\end{lem}
\begin{proof}
	Let us first show that there exists $\Lip(\e)>0$ such that
	\be\label{eq:K-Lip}
		|\grad_{\Sph^\dim} K_y(x) - \grad_{\Sph^\dim} K_z(x)| \leq \Lip(\e) d(y,z), \qquad \mbox{for all $x,y,z \in \S$}.
	\ee
	Let $x,y,z \in \S$ and write $\theta_{xy} := d(x,y)$ and $\theta_{xz} := d(x,z)$. Then, by \eqref{eqn:gradd2-sphere},
	\begin{align}
		& |\grad_{\Sph^\dim} d^2_y(x) - \grad_{\Sph^\dim} d^2_z(x)| =2  |f(\theta_{xy}) (\cos(\theta_{xy})x - y) - f(\theta_{xz})(\cos(\theta_{xz})x - z)| \nonumber\\
		&\qquad \leq 2 |(f(\theta_{xy})-f(\theta_{xz}))(\cos(\theta_{xy})x - y)| + 2 |f(\theta_{xz})((\cos(\theta_{xy}) - \cos(\theta_{xz}))x| + 2 |f(\theta_{xz})(z-y)|\nonumber \\
		&\qquad \leq 4 L_f(\e) |\theta_{xy} - \theta_{xz}| + 2 C_f(\e) |\theta_{xy} - \theta_{xz}| + 2 C_f(\e) |z-y| \nonumber \\
		& \qquad \leq 4 (L_f(\e) + C_f(\e))d(y,z),
        \label{eqn:est-d2-mod}
	\end{align}
	where we added and subtracted $f(\theta_{xz})(\cos(\theta_{xy})x-y)$ on the first line and then used triangle inequality, we used the bound and Lipschitz constant of $f$ for the second inequality sign, and finally, we used $|\theta_{xy} - \theta_{xz}| \leq d(y,z)$ by triangle inequality, and the fact that the Euclidean distance $|z-y|$ is smaller than the spherical distance $d(y,z)$. 
	
Now compute
	\begin{align}
		& |\grad_{\Sph^\dim} K_y(x) - \grad_{\Sph^\dim} K_z(x)| = |g'(d(x,y)^2) \grad_{\Sph^\dim} d^2_y(x) - g'(d(x,z)^2) \grad_{\Sph^\dim} d^2_z(x)| \nonumber \\
		&\qquad \leq |g'(\theta_{xy}^2) - g'(\theta_{xz}^2)| |\grad_{\Sph^\dim} d^2_y(x)| + |g'(\theta_{xz}^2)| |\grad_{\Sph^\dim} d^2_y(x) - \grad_{\Sph^\dim} d^2_z(x)| \nonumber \\
		&\qquad \leq 2L_{g'}(\e) |\theta_{xy}+\theta_{xz}| |\theta_{xy}-\theta_{xz}| \theta_{xy}+ 4 C_{g'}(\e) (L_f(\e) + C_f(\e)) d(y,z) \nonumber \\
		&\qquad \leq (4\pi^2 L_{g'}(\e)  + 4 C_{g'}(\e) (L_f(\e) + C_f(\e)))d(y,z).
		\label{eqn:est-K-mod}	
	\end{align}
In the above, we first added and subtracted $g'(\theta_{xz}^2) \grad_{\Sph^\dim} d^2_y(x)$ on the first line and used triangle inequality. For the second inequality we used \eqref{eqn:est-d2-mod}, the bound and Lipschitz constant of $g'$, and $|\grad_{\Sph^\dim} d^2_y(x)|= 2 \theta_{xy}$. For the last inequality we used $|\theta_{xy} - \theta_{xz}| \leq d(y,z)$ by triangle inequality, and that $\theta_{xy},\theta_{xz} <\pi$.

By setting 
\[
\Lip(\e):= 4 \pi^2 L_{g'}(\e)  + 4 C_{g'}(\e) (L_f(\e) + C_f(\e)),
\]
we get \eqref{eq:K-Lip}. Then, for $(x,t) \in \S \times [0,T)$ arbitrary fixed, take $\pi_t \in \Pi(\rho_t,\sigma_t)$ to be an optimal transport plan between $\rho_t$ and $\sigma_t$, and estimate:
	\begin{align*}
		|\V[\rho](x,t) - \V[\sigma](x,t)| &= \left| \int_{\S} \grad_{\Sph^\dim}K_y(x)\d\rho_t(y) - \int_{\S} \grad_{\Sph^\dim}K_z(x) \d\sigma_t(z) \right|\\
		&= \left| \int_{\S \times \S} \grad_{\Sph^\dim} K_y(x) \d\pi_t(y,z) - \int_{\S \times \S} \grad_{\Sph^\dim} K_z(x) \d\pi_t(y,z) \right|\\
		&\leq \int_{\S \times \S} |\grad_{\Sph^\dim} K_y(x) - \grad_{\Sph^\dim} K_z(x)| \d\pi_t(y,z).
	\end{align*}
Hence, using \eqref{eqn:est-K-mod},
	\begin{align}\label{est:LipX}
		|\V[\rho](x,t) - \V[\sigma](x,t)| &\leq \Lip(\e) \int_{\S \times \S} d(y,z) \d\pi_t(y,z)
= \Lip(\e)W_1(\rho_t,\sigma_t)\\
	&\leq \Lip(\e) \bd_1(\rho,\sigma) \nonumber.
	\end{align}
	Taking now the supremum in $(x,t)\in \S \times [0,T)$ on the left-hand side above gives the result.
\end{proof}

The main result of this section is given by the following theorem. The proof is based on a fixed point argument, borrowing from the layout and the general technique used by Canizo {\em et al.} \cite{CanizoCarrilloRosado2011} to prove the well-posedness of solutions in the Euclidean case. We point out that we work here with probability measures on $\S$, which is a geodesically convex set that can be covered by a single chart. It is expected in this case that the gradient flow techniques used in \cite{WuSlepcev2015} can be extended to deal with intrinsic interactions as in our setup. Nevertheless, this extension has not been worked out explicitly in the literature.

\begin{thm}[Well-posedness on open hemisphere]\label{thm:well-posedness}
	Suppose that $K$ satisfies \ref{hyp:K} and let $\rho_0 \in \P(\S)$. Then, there exist $T>0$ and a unique weak solution among curves in $\Cont([0,T);\P(\S))$ to the aggregation model \eqref{eqn:model} starting from $\rho_0$.
	\end{thm}
\begin{proof} We first invoke some results included in Appendix. Specifically, by Lemma \ref{lem:interaction-complete} the interaction velocity field $\V[\sigma]$ (with $\sigma$ fixed) is locally Lipschitz and hence it satisfies the assumptions of the local well-posedness result in Theorem \ref{thm:Cauchy-Lip}. In addition, by Remark \ref{rem:indep-max-time}, the maximal time of existence for its flow map does not depend on $\sigma$. Consequently, there exists a maximal time $\tau>0$ such that the map $\Gamma$, given by
	\begin{equation}
	\label{eqn:Gamma}
		\Gamma(\sigma)(t) = {\Psi_{\V[\sigma]}^t}\# \rho_0, \qquad \mbox{for all $\sigma \in \Cont([0,\tau);\P(\S))$ and $t \in [0,\tau)$},
	\end{equation}
	is well-defined, where $\Psi_{v[\sigma]}$ is the unique flow map generated by $(v[\sigma],\supp(\rho_0))$ and defined on $\supp(\rho_0)\times [0,\tau)$. We will prove that $\Gamma$ is a map from $\Cont([0,\tau);\P(\S))$ into itself and that it has a unique fixed point, which directly shows the desired result.
	
Let us show first that $\Gamma$ maps $\Cont([0,\tau);\P(\S))$ into itself. To this end, fix $\sigma \in \Cont([0,\tau);\P(\S))$. By the proof of Theorem \ref{thm:Cauchy-Lip} we know that $\Psi_{v[\sigma]}^t(x)\in \S$ for all $x\in\supp(\rho_0)$ and $t\in[0,\tau)$, so that $\Gamma(\sigma)(t)$ is supported in $\S$. We have in fact $\Gamma(\sigma)(t) \in\P(\S)$ for all $t\in[0,\tau)$ since $\rho_0\in \P(\S)$ and the push-forward conserves mass. Moreover, we get that the map $t \to \Gamma(\sigma)(t)$ is continuous due to Lemmas \ref{lem:dist-flow-maps-K} and \ref{lem:preliminary}\ref{it:prel2}. All in all we obtain $\Gamma \: (\Cont([0,\tau);\P(\S)),\bd_1) \to (\Cont([0,\tau);\P(\S)),\bd_1)$.


We now show that $\Gamma$ is a contraction if we restrict our final time to some $T\leq \tau$ to be determined. Let $\rho, \sigma \in \Cont([0,\tau);\P(\S))$. Then, for all $t \in [0,\tau)$,
	\begin{align}
		W_1({\Psi_{\V[\rho]}^t}\#\rho_0,{\Psi_{\V[\sigma]}^t}\#\rho_0) &\leq \sup_{x \in\supp(\rho_0)} d(\Psi_{\V[\rho]}^t(x),\Psi_{\V[\sigma]}^t(x)) \nonumber \\
		&\leq C(\e,t) \| \V[\rho] - \V[\sigma]\|_{L^\infty([0,\tau)\times \S)} \nonumber \\[5pt]
		&  \leq C(\e,t) \Lip(\e) \bd_1(\rho,\sigma),
		\label{eqn:estW1}
	\end{align}
	where for the first inequality we used Lemma \ref{lem:preliminary}\ref{it:prel1}, for the second inequality we used Lemmas \ref{lem:dist-flow-maps-K} and \ref{lem:dist-flow-maps} with
	\bes
		C(\e,t) =  \frac{e^{(L(\e) +  4\pi C_{g'}(\e))t}-1}{L(\e) + 4\pi C_{g'}(\e)},
	\ees
and for the last inequality we used Lemma \ref{lem:grad-lip-2}. Since $C(\e,t)$ is increasing in $t$, with $\lim_{t \to 0} C(\e,t) = 0$ and $\Lip(\e)$ is independent of time, we can choose $T\leq \tau$ small enough so that 
\[
C(\e,t) \Lip(\e) < C(\e,T) \Lip(\e) < \overline{C}(\e), \qquad \text{ for all } t \in [0,T),
\]
for some constant $\overline{C}(\e)<1$. Restricting $T$ accordingly, by taking the supremum over $[0,T)$ in \eqref{eqn:estW1} we find:
	\bes
		\bd_1(\Gamma(\rho),\Gamma(\sigma)) \leq \overline{C}(\e) \bd_1(\rho,\sigma),
	\ees
with $\overline{C}(\e)<1$. This shows that the restriction of $\Gamma$ to $(\Cont([0,T);\P(\S)),\bd_1)$ is a contraction.

We have thus shown that $\Gamma\: (\Cont([0,T);\P(\S)),\bd_1) \to (\Cont([0,T);\P(\S)),\bd_1)$ has a unique fixed point, that is, there exists a unique $\rho \in \Cont([0,T);\P(\S))$ such that
	\bes
		\rho_t = \Psi_{\V[\rho]}^t \# \rho_0 \quad \mbox{for all $[0,T)$},
	\ees
which means that $\rho$ is the desired solution.
\end{proof}

\begin{rem}
\label{rem:global} The solution established in Theorem \ref{thm:well-posedness} can be extended in time as long as its support remains within the set $\S$. In the particular case of purely attractive interactions ($g'\geq 0$), we show in Proposition \ref{prop:inv-cont} below that the well-posedness of solutions holds globally in time, i.e., $T=\infty$; in other words, $\S$ is an invariant set for the dynamics. Moreover, with further assumptions on the interaction potential, solutions approach asymptotically a consensus state; see Theorem \ref{thm:consensus-cont}.
\end{rem}

\subsection{Stability and particle solutions}
\label{subsect:stability}
In this section we investigate the stability of solutions to \eqref{eqn:model} with respect to the initial conditions and, based on it, we demonstrate the mean-field approximation. The following result is analogous to \cite[Theorem 3.16]{CanizoCarrilloRosado2011}.

\begin{thm}[Stability]\label{thm:stability}
	Consider an interaction potential $K$ that satisfies \ref{hyp:K}. Let $\rho_0,\sigma_0 \in \P(\S)$, and $\rho$ and $\sigma$ be weak solutions to \eqref{eqn:model} defined on $[0,T)$ starting from $\rho_0$ and $\sigma_0$, respectively. Then, there exist $T^*\in(0,T)$ and an increasing, bounded function $r(\e,\cdot)$ with $r(\e,0) = 1$ such that
\bes
	W_1(\rho_t,\sigma_t) \leq r(\e,t)W_1(\rho_0,\sigma_0),	\qquad \mbox{for all $t\in [0,T^*)$}.
\ees
\end{thm}
\begin{proof}
	Let $\Sigma = \supp(\rho_0)\cup\supp(\sigma_0)$. By compactness of $\Sigma$, Theorem \ref{thm:Cauchy-Lip} and Lemma \ref{lem:interaction-complete}, we know the existence of unique maximal flow maps $\tilde \Psi_{v[\rho]}$ and $\tilde \Psi_{v[\sigma]}$ generated by $(v[\rho],\Sigma)$ and $(v[\sigma],\Sigma)$, respectively. Call $\tau_\rho>0$ and $\tau_\sigma>0$ the respective maximal times of existence, and set $T^* = \min(\tau_\rho,\tau_\sigma,T)$. Fix $t\in[0,T^*)$. Since $\rho$ and $\sigma$ are weak solutions up to time $T$, we have, by the triangle inequality,
	\begin{align}
	\label{est:stab1}
		W_1(\rho_t,\sigma_t) &= W_1(\tilde \Psi_{\V[\rho]}^t \# \rho_0,\tilde \Psi_{\V[\sigma]}^t \# \sigma_0)\nonumber \\[5pt]
		&\leq W_1(\tilde \Psi_{\V[\rho]}^t \# \rho_0,\tilde \Psi_{\V[\sigma]}^t \# \rho_0) + W_1(\tilde \Psi_{\V[\sigma]}^t \# \rho_0,\tilde \Psi_{\V[\sigma]}^t \# \sigma_0).
	\end{align}
By Lemma \ref{lem:dist-flow-maps-K}, the vector field $\V[\sigma]$ is bounded and Lipschitz continuous with respect to its first variable, which in turn implies by Lemma \ref{lem:Lipschitz-initial} that the map $\tilde \Psi_{\V[\sigma]}^t$ is Lipschitz continuous on $\S$ with Lipschitz constant $e^{(L(\e)+2\| \V[\sigma]\|_{L^\infty})t}$, where we write $\|\V[\sigma]\|_{L^\infty}$ for $\|\V[\sigma]\|_{L^\infty(\S\times[0,T))}$ for the rest of the proof. Using Lemma \ref{lem:preliminary} (parts \ref{it:prel1} and \ref{it:prel3}) for the first and second terms in the right-hand side of \eqref{est:stab1}, we further estimate:
\begin{align}
\label{est:stab2}
W_1(\tilde \Psi_{\V[\rho]}^t \# \rho_0, & \tilde \Psi_{\V[\sigma]}^t \# \rho_0)  + W_1(\tilde \Psi_{\V[\sigma]}^t \# \rho_0,\tilde \Psi_{\V[\sigma]}^t \# \sigma_0) \nonumber \\[5pt]
& \leq \sup_{x \in\supp(\rho_0)} d(\tilde \Psi_{\V[\rho]}^t(x),\tilde \Psi_{\V[\sigma]}^t(x)) + e^{(L(\e)+2\| \V[\sigma]\|_{L^\infty})t} W_1(\rho_0,\sigma_0).
\end{align}
Also, using estimate \eqref{eqn:est-total} for the vector fields $\V[\rho]$ and $\V[\sigma]$, and integrating it with an integrating factor, we find for all $x \in \supp(\rho_0)$,
\begin{align}
\label{est:stab3}
d(\tilde \Psi_{\V[\rho]}^t(x),\tilde \Psi_{\V[\sigma]}^t(x)) & \leq \int_0^t e^{(L(\e)+2\|\V[\sigma]\|_{L^\infty})(t-s)} \| \V[\rho](\cdot,s) - \V[\sigma](\cdot,s) \|_{L^\infty(\S)} \d s \nonumber \\
&\leq \Lip(\e) \int_0^t e^{(L(\e)+2\|\V[\sigma]\|_{L^\infty})(t-s)} W_1(\rho_s,\sigma_s) \d s,
\end{align}
where for the second inequality we used \eqref{est:LipX}.

Combine \eqref{est:stab1}, \eqref{est:stab2} and \eqref{est:stab3} to find, after multiplying by $e^{-(L(\e)+2\|\V[\sigma]\|_{L^\infty})t}$:
	\bes
		e^{-(L(\e)+2\|\V[\sigma]\|_{L^\infty})t} W_1(\rho_t,\sigma_t) \leq \Lip(\e) \int_0^t e^{- (L(\e)+2\|\V[\sigma]\|_{L^\infty}) s} W_1(\rho_s,\sigma_s) \d s + W_1(\rho_0,\sigma_0).
	\ees
By Gronwall's lemma the above estimate yields
	\bes
		e^{-(L(\e)+2\|\V[\sigma]\|_{L^\infty})t} \, W_1(\rho_t,\sigma_t) \leq e^{\Lip(\e)t} \, W_1(\rho_0,\sigma_0).
	\ees
Finally, use the upper bound for $\|\V[\sigma]\|_{L^\infty}$ established in Lemma \ref{lem:dist-flow-maps-K}, and set
\begin{equation}
\label{eqn:r}
r(\e,t):= e^{(\Lip(\e)+L(\e)+ 4 \pi C_{g'}(\e))t}
\end{equation}
to arrive at the desired conclusion.
\end{proof}

An important application of the stability result is the approximation of a continuum measure by empirical measures, referred to as mean-field approximation. We investigate this below.

\begin{lem}\label{lem:atomic-sol}
	Suppose that $K$ satisfies \ref{hyp:K}. Take $\N$ a positive integer and consider a collection of masses $(m_i)_{i=1}^{\N} \subset (0,1)$ such that $\sum_{i=1}^{\N} m_i = 1$, and points $(x_{i}^0)_{i=1}^{\N} \subset \S$. Then, there exists $T>0$ and a unique collection of trajectories $(x_i)_{i=1}^{\N}$ so that  $x_i\:[0,T) \to \S$ satisfies, for all $i\in \{1,\dots,n\}$ and $t\in[0,T)$,
	\be\label{eq:characteristics-particles}
		\begin{cases}
			x_i'(t) = \V[\rho^{\N}](x_i(t),t),\\
			x_i(0) = x_{i}^0,
		\end{cases}
	\ee
	where $\rho^{\N}\:[0,T) \to \P(\S)$ is the empirical measure associated to masses $m_i$ and trajectories $x_{i}$, for $i\in\{1,\dots,\N\}$, i.e.,
	\be\label{eq:atomic}
		\rho_t^{\N} = \sum_{i=1}^{\N} m_i \delta_{x_i(t)}, \qquad \mbox{for all $t\in [0,T)$}.
	\ee
	Furthermore, $\rho^{\N}$ is the unique weak solution to \eqref{eqn:model} on $[0,T)$ with initial data 
	\begin{equation}
	\label{eq:atomic-initial}
	\rho^{\N}_0 =  \sum_{i=1}^{\N} m_i \delta_{x_{i}^0}.
	\end{equation}
\end{lem}
\begin{proof} 
For all $i\in\{1,\dots,n\}$ we can rewrite the first equation in \eqref{eq:characteristics-particles} as
\bes
	x_i'(t) = - \sum_{j=1}^{\N} m_i\nabla_{\Sph^\dim}K_{x_j(t)}(x_i(t)).
\ees
The well-posedness of solutions to \eqref{eq:characteristics-particles} thus follows from Theorem \ref{thm:Cauchy-Lip} and the local Lipschitz continuity on charts of $x\mapsto \grad_{\Sph^\dim} K_z(x)$, uniformly in $z\in\S$, as discussed in the proof of Lemma \ref{lem:interaction-complete}.
Since $x_i$, $i=1,\dots,\N$, satisfies the first-order ODE system \eqref{eq:characteristics-particles}, $x_i$ is continuous. Let $\tf\in\Cont_\mt{b}(\S)$ and let $t\in[0,T)$ and $(t_k)_{k \geq 1} \subset [0,T)$ be such that $t_k\to t$ as $k\to\infty$. Then, using that $\tf$ is bounded,
	\bes
		\int_\S \tf(x) \d\rho_{t_k}^{\N}(x) = \sum_{i =1}^{\N} m_i \tf(x_i(t_k)) \to \sum_{i =1}^{\N} m_i \tf(x_i(t)) = \int_\S \tf(x) \d\rho_t^{\N}(x), 
	\ees
	which shows that $\rho^{\N} \in \Cont([0,T);\P(\S))$. Let $\Psi^t_{\V[\rho^{\N}]}$ be the unique flow map generated by $\V[\rho^{\N}]$ defined on $\supp(\rho_0^n)\times [0,T)$ and let $\zeta \: \S \to [-\infty,\infty]$ be measurable such that $\zeta \circ \Psi^t_{\V[\rho^{\N}]}$ is integrable with respect to $\rho_0^n$. Then $x_i(t) = \Psi^t_{\V[\rho^{\N}]}(x_i^0)$ for all $i\in\{1,\dots,n\}$ and $t\in[0,T)$, and we get
	\begin{align*}
		\int_\S \zeta(x) \d(\Psi^t_{\V[\rho^{\N}]} \#\rho_0^{\N})(x) &= \int_\S \zeta(\Psi^t_{\V[\rho^{\N}]}(x)) \d \rho_0^{\N}(x) 
		= \sum_{i=1}^{\N} m_i \zeta(\Psi^t_{\V[\rho^{\N}]}(x_i^0))\\
		&= \sum_{i=1}^{\N} m_i \zeta(x_i(t)) = \int_\S \zeta(x) \d\rho_t^{\N}(x),
	\end{align*}
	which proves that 
	\bes
		\rho_t^{\N} = \Psi^t_{\V[\rho^{\N}]} \#\rho_0^{\N}, \qquad \mbox{for all $t\in[0,T)$}.
	\ees
	Thus, $\rho^{\N}$ is a weak solution to \eqref{eqn:model} on $[0,T)$, with initial datum $\rho^{\N}_0$. The uniqueness of $\rho^{\N}$ follows directly from Theorem \ref{thm:well-posedness}.
\end{proof}

\begin{thm}[Mean-field limit]\label{thm:mfl}
	Suppose that $K$ satisfies \ref{hyp:K}. Let $\rho_0 \in \P(\S)$ and let $(\rho_0^{\N})_{\N \in\mathbb{N}} \subset\P(\S)$ be of the form \eqref{eq:atomic-initial} and such that
	\bes
		W_1(\rho_0^{\N},\rho_0) \to 0, \qquad \mbox{ as $\N\to\infty$}.
	\ees
	Suppose furthermore that $T>0$ is such that there exist a unique weak solution $\rho$ to \eqref{eqn:model} on $[0,T)$ starting from $\rho_0$ and a unique weak solution $\rho^{\N}$ to \eqref{eqn:model} on $[0,T)$ starting from $\rho_0^{\N}$ for all $\N\in\mathbb{N}$, which we know is of the form \eqref{eq:atomic}. Then,  there exists $T^*\in(0,T)$ such that
	\bes
		 \sup_{t\in[0,T^*)} W_1(\rho_t^{\N},\rho_t)  \to 0, \qquad \mbox{ as $\N\to\infty$}.
	\ees
\end{thm}
\begin{proof}
	By Theorem \ref{thm:stability}, there exists a strictly increasing, bounded function $r_{\N}(\e,\cdot) \: [0,T^*) \to[0,\infty)$ for all $\N\in \mathbb{N}$ such that
	\bes
		W_1(\rho_t^{\N},\rho_t) \leq r_{\N}(\e,t) W_1(\rho_0^{\N},\rho_0), \qquad \mbox{for all $t\in[0,T^*)$}.
	\ees
	As one can see from the proof of Theorem \ref{thm:stability} (see equation \eqref{eqn:r}), the function $r_{\N}(\e,\cdot)$ is independent of $\N$, hence we rename it $r(\e,\cdot)$. By boundedness of $r(\e,\cdot)$ on $[0,T^*)$ (call $C_r(\e,T^*)>0$ such a bound), we get
	\bes
		\sup_{t\in[0,T^*)} W_1(\rho_t^{\N},\rho_t) \leq C_r(\e,T^*) W_1(\rho_0^{\N},\rho_0) \to 0, \qquad \mbox{ as $\N\to\infty$},
	\ees
	which finishes the proof.
\end{proof}


\section{Asymptotic behaviour on sphere}
\label{sect:synchronization}

In this section we study the asymptotic dynamics of model \eqref{eqn:model} on sphere $\Sph^\dim$ when the interaction potential is purely attractive. As in Section \ref{sect:sphere} we equip $\Sph^\dim$ with the metric induced by the canonical topology of the ambient Euclidean space $\R^{\dim+1}$. The aim is to investigate the formation of consensus (or synchronized) states asymptotically, i.e., when the solution $\rho_t$ approaches a Dirac mass on sphere as $t \to \infty$. 

Throughout this section the interaction potential $K$ is assumed to satisfy Assumption \ref{hyp:K} with $g$ nondecreasing (i.e., $g' \geq 0$), which corresponds to {\em purely attractive} interaction forces. 


\subsection{Global well-posedness and geodesic disks as invariant sets}
\label{subsect:geod-disks}
In consideration of the well-posedness result in Theorem \ref{thm:well-posedness}, a key aspect for investigating the asymptotic dynamics is whether the solution remains supported on the set $\S$, given in \eqref{eqn:setS} for some $\e\in(0,\pi/2)$, during the whole time evolution. In other words, we want to get global versions of Theorem \ref{thm:well-posedness} and Lemma \ref{lem:atomic-sol}. We study this below.

Fix $\e\in(0,\pi/2)$ throughout. Note that the set $\S$ is a geodesic disk of radius $\pi/2-\e$ centred at the North pole of the unit sphere. In general, the geodesic disk on sphere with centre at the North pole $\North$ and radius $r>0$ is given by:
\[
D_r=\{x\in\mathbb{S}^{\dim} \st d(x, \North)<r\}.
\]
Given that the geodesics on sphere lie on great circles, all disks $D_r$ with $0<r\leq \pi/2$, and their closures $\overline D_r$, are geodesically convex; in particular, so is $\S$. We observe that by spherical symmetry, the results we prove below are easily extended to any centre which is not the North pole. We start with the continuum model \eqref{eqn:model}. 

\begin{prop}[Global well-posedness in continuum model]
\label{prop:inv-cont}
Let $K$ satisfy \ref{hyp:K} with $g'\geq0$. Let $\rho_0 \in \P(\S)$ be such that $\supp(\rho_0)\subset \overline D_r$ for some $r<\pi/2-\e$. Then, there exists a unique global weak solution to the aggregation model \eqref{eqn:model} in $\Cont([0,\infty);\P(\S))$ that starts from $\rho_0$; moreover, $\supp(\rho_t) \subset \overline D_r$ for all $t\in[0,\infty)$.
\end{prop}
\begin{proof} We will use the global version of the Cauchy-Lipschitz theorem presented in the Appendix; see Theorem \ref{thm:global-Cauchy-Lip} and also Lemma \ref{lem:interaction-complete-global} for its application to the interaction velocity field. By abuse of notation, let us write $\P(\overline D_r)$ for the set of Borel probability measures on $\S$ which are supported within $\overline D_r$. By Theorem \ref{thm:global-Cauchy-Lip} and Lemma \ref{lem:interaction-complete-global}, the map
\bes
	\Gamma(\sigma)(t) = \Psi_{v[\sigma]}^t \# \rho_0, \quad \mbox{for all $\sigma \in \Cont([0,\infty);\P(\overline D_r))$ and $t\in[0,\infty)$},
\ees
is well-defined, where $\Psi_{v[\sigma]}^t$ is the unique global flow map generated by $(\V[\sigma],\supp(\rho_0))$. By following the same approach as in the proof of Theorem \ref{thm:well-posedness}, we get that $\Gamma$ is a map from $(\Cont([0,\infty);\P(\overline D_r)),\bd_1)$ into itself. We also get the existence of a time $T>0$ and a constant $\overline C \in (0,1)$ such that the restriction of $\Gamma$ to $(\Cont([0,T);\P(\overline D_r)),\bd_1)$ is a contraction, which means that there exists a unique $\rho \in \Cont([0,T);\P(\overline D_r))$ such that
	\bes
		\rho_t = \Psi_{\V[\rho]}^t \# \rho_0 \quad \mbox{for all $[0,T)$}.
	\ees
From the proof of Theorem \ref{thm:well-posedness} we note that the time $T$ is independent of $\rho_0$. Therefore, we can iteratively patch solutions together continuously through time to get the existence of a unique weak solution among curves in $\Cont([0,\infty);\P(\overline D_r))$, which concludes the proof.
\end{proof}


We now get the analogous result of Proposition \ref{prop:inv-cont} for the discrete model \eqref{eq:characteristics-particles}:

\begin{prop}[Global well-posedness in discrete model]
\label{prop:inv-disc}
Let $K$ satisfy \ref{hyp:K} with $g'\geq0$. Take $\N$ to be a positive integer and consider a collection of masses $(m_i)_{i=1}^{\N} \subset (0,1)$ such that $\sum_{i=1}^{\N} m_i = 1$, and points $(x_{i}^0)_{i=1}^{\N} \subset \overline D_r$ for some $r<\pi/2-\e$. Then, there exists a unique global collection of trajectories $(x_i)_{i=1}^{\N}$ that satisfies, for all $i\in \{1,\dots,n\}$ and $t\in[0,\infty)$, $x_i(t)\in \overline D_r$ and
\be\label{eqn:model-disc}
	\begin{cases} x_i'(t) = -\displaystyle \sum_{j=1}^{\N} m_i \nabla_{\Sph^\dim}K_{x_j(t)}(x_i(t)),\\ x_i(0) = x_i^0. \end{cases}
\ee
\end{prop}
\begin{proof}
The local well-posedness follows as in the proof of Lemma \ref{lem:atomic-sol}, while the global extension follows directly by applying Theorem \ref{thm:global-Cauchy-Lip} and the fact that 
\bes
	\log_x \North \cdot \nabla_{\Sph^\dim}K_{y}(x) \leq 0 \quad \mbox{for all $x\in \S\setminus D_r$ and $y\in \overline D_r$},
\ees
as can be inferred from the proof of Lemma \ref{lem:interaction-complete-global}.
\end{proof}

\begin{rem}
	In the terminology of dynamical systems theory, Propositions \ref{prop:inv-cont} and \ref{prop:inv-disc} show that any closed disk in $\S$ is an invariant set for the aggregation dynamics given by \eqref{eqn:model} and \eqref{eqn:model-disc}, respectively.
\end{rem}


\subsection{Asymptotic consensus in the continuum model}
\label{subsect:consensus-cont}
We consider the asymptotic behaviour in the continuum model. Specifically, we study the formation of consensus by investigating the behaviour of an energy functional. 

As discussed in \cite{FeZh2019}, model \eqref{eqn:model} is a gradient flow with respect to an energy functional. For the model set up on $\S$, this energy functional $E\:\P(\S)\to\R$ is given by:
\begin{equation}
\label{eqn:energy-cont}
	E[\rho]= \frac{1}{2}\int_{\S} \int_{\S} K(x, y)\d\rho(x)\d\rho(y), \qquad \mbox{for all $\rho\in\P(\S)$}.
\end{equation}
Because $K$ is assumed to satisfy \ref{hyp:K}, it is bounded and therefore $E$ is indeed well-defined on $\P(\S)$. To simplify notation, given a weak solution $\rho$ to \eqref{eqn:model} defined on $[0,\infty)$ which is clear from context, we shall write $t\mapsto E(t)$ the map given by $E(t) = E[\rho(t)]$ for all $t\in[0,\infty)$ and by $\V\:\S\times [0,\infty)$ the function $\V(x,t) = \V[\rho](x,t)$ for all $(x,t)\in\S\times[0,\infty)$, where we recall that $\V[\rho]$ is the interaction velocity field defined in \eqref{eqn:v-field}.

First, we would like to show that any global weak solution $\rho$ to \eqref{eqn:model} starting inside a closed disk $\overline D_r$ with $r<\pi/2-\e$ satisfies
\bes
	\lim_{t\to\infty} \int_\S |\V(x,t)|^2 \d\rho_t(x) = 0.
\ees
To this end, we will apply Barbalat's lemma \cite{Barbalat1959} which means we need to show that $E(t)$ has a finite limit as $t\to\infty$ and $t\mapsto E''(t)$ is bounded on $[0,\infty)$.

\begin{lem}\label{lem:E-lim}
	Let $K$ satisfy \ref{hyp:K} with $g'\geq0$, and let $\rho_0\in\P(\S)$ be such that $\supp(\rho_0)\subset \overline D_r$ for some $r<\pi/2-\e$. Write $\rho\in\Cont([0,\infty);\P(\S))$ the global weak solution to \eqref{eqn:model} starting from $\rho_0$ from Proposition \ref{prop:inv-cont}. Then $E(t)\to E_\infty$ as $t\to\infty$ for some $E_\infty\in\R$.
\end{lem} 
\begin{proof}
	Writing $\Psi_\V$ for the global flow map generated by $(\V,\supp(\rho_0))$ and using the push-forward formulation of $\rho$ and the chain rule, one can compute, for all $t\in[0,\infty)$,
\begin{align}
\label{eqn:dEdt-cont}
	E'(t)&= \frac12 \frac{\der}{\der t} \int_{\S}\int_\S K(\Psi_{\V}^t(x),\Psi_{\V}^t(y)) \d\rho_0(x)\d\rho_0(y) \nonumber\\
	&=  \int_{\S} \grad_{\Sph^\dim}K * \rho_t(\Psi_{\V}^t(x)) \cdot \V(\Psi_{\V}^t(x),t) \d\rho_0(x) \nonumber\\
	& = -  \int_{\S} | \V(\Psi_{\V}^t(x),t)|^2 \d\rho_0(x) \nonumber\\
	&= - \int_\S |\V(x,t)|^2 \d\rho_t(x) \leq 0,
\end{align}
where for the second equality we used the symmetry of $K$. Note that the last term in \eqref{eqn:dEdt-cont} is well-defined and bounded by Lemma \ref{lem:dist-flow-maps-K}.

Proposition \ref{prop:inv-cont} ensures that the global solution $\rho$ satisfies $\supp(\rho_t)\subset\overline D_r$, and since $\overline D_r$ is compact, the map $t\mapsto E(t)$ is bounded below (because $K$ is bounded on compact sets). Moreover, $t\mapsto E(t)$ is nonincreasing by \eqref{eqn:dEdt-cont} and we thus conclude $E(t) \to E_\infty$ as $t\to\infty$ for some $E_\infty \in\R$.
\end{proof}

\begin{lem}
\label{lem:d2E_dt2}
	Let $K$ satisfy \ref{hyp:K} and $r<\pi/2-\e$. Suppose that $g'\geq0$ and $g'$ is continuously differentiable on $[0,4r^2]$. Let $\rho_0\in\P(\S)$ be such that $\supp(\rho_0)\subset \overline D_r$, and write $\rho\in\Cont([0,\infty),\P(\S))$ the global weak solution to \eqref{eqn:model} starting from $\rho_0$ from Proposition \ref{prop:inv-cont}. Then $E''$ is bounded on $[0,\infty)$.
\end{lem}
\begin{proof}
By Proposition \ref{prop:inv-cont} we know that $\supp(\rho_t)\subset \overline D_r$ for all $t\in[0,\infty)$. We write $\Psi_\V$ for the global flow map generated by $(\V,\supp(\rho_0))$, which satisfies $\Psi_\V^t(x)\in \overline D_r$ for all $x\in\supp(\rho_0)$ and $t\in[0,\infty)$. We know $E''$ exists by continuity of $g''$ on $[0,4r^2]$, and from the computation in \eqref{eqn:dEdt-cont} we have, for all $t\in[0,\infty)$,
\begin{align}
\label{eq:E-second}
	E''(t) &= - \frac{\der}{\der t}  \int_{\S} | \V(\Psi_{\V}^t(x),t)|^2 \d\rho_0(x) \nonumber   \\
	&= -2 \int_\S \frac{\der}{\der t} \V(\Psi_{\V}^t(x),t) \cdot  \V(\Psi_{\V}^t(x),t) \d\rho_0(x).
\end{align}
By definition of $v$ and the formulation of push-forward, we find:
\be
\label{eq:E-second2}
\frac{\der}{\der t} \V(\Psi_{\V}^t(x),t) = - \int_\S  \frac{\der}{\der t} \grad_{\Sph^\dim} K_{\Psi_{\V}^t(y)}(\Psi_{\V}^t(x)) \d\rho_0(y).
\ee
To apply the product rule to compute the integrand above, set the following notation for $\grad_{\Sph^\dim}K_y(x)$ when one of the variables is fixed and the other changes:
\bes
u_y(x) := \grad_{\Sph^\dim}K_y(x), \quad \mbox{and} \quad w_x(y) := \grad_{\Sph^\dim}K_y(x).
\ees
Then, by the product and chain rules, the above integrand becomes
\begin{equation}
\label{eqn:prod-rule}
\frac{\der}{\der t} \grad_{\Sph^\dim} K_{\Psi_{\V}^t(y)}(\Psi_{\V}^t(x))  =  \der u_{\Psi_\V^t(y)}(\Psi_\V^t(x))(v(\Psi_\V^t(x),t)) + \der w_{\Psi_\V^t(x)}(\Psi_\V^t(y))(v(\Psi_\V^t(y),t)).
\end{equation}
Using the form of the interaction potential given in \ref{hyp:K}, we have, for all $x,y\in\S$ and $\alpha\in T_x\Sph^\dim$,
\bes
	\der u_y(x)(\alpha) = \Hess_{\Sph^\dim} K_y(x) \alpha = g''(d(x,y)^2) \ap{\grad_{\Sph^\dim} d_y^2(x),\alpha}_x \grad_{\Sph^\dim} d_y^2(x) + g'(d(x,y)^2) \Hess_{\Sph^\dim} d_y^2(x) \alpha,
\ees
where $\Hess_{\Sph^\dim}$ stands for the Hessian operator on the manifold $\Sph^\dim$. Also, for all $\beta\in T_y\Sph^\dim$, 
\bes
	\der w_x(y)(\beta) = g''(d(x,y)^2) \ap{\grad_{\Sph^\dim}d_x^2(y), \beta}_y \grad_{\Sph^\dim} d_y^2(x) - 2g'(d(x,y)^2) \der \log_x(y)(\beta).
\ees

As $g'$ and $g''$ are continuous on $[0,4r^2]$, the maps $(x,y)\mapsto g'(d(x,y)^2)$ and $(x,y)\mapsto g''(d(x,y)^2)$ are bounded on the compact set $\overline D_r \times\overline D_r$. Furthermore, by smoothness of the manifold $\Sph^\dim$, the map $(x,y)\mapsto d_y^2(x)$ is smooth on the geodesically convex set $\overline D_r\times \overline D_r$. This  implies that $(x,y)\mapsto \grad_{\Sph^\dim} d_y^2(x)$, $(x,y)\mapsto \Hess_{\Sph^\dim} d_y^2(x)$ and $(x,y) \mapsto \der\log_x(y)$ are bounded on $\overline D_r\times \overline D_r$, from which we get that $(x,y)\mapsto \der u_y(x)$ and $(x,y)\mapsto \der w_x(y)$ are bounded on $\overline D_r\times \overline D_r$. Then, as the map $(x,t)\mapsto \V(x,t)$ is bounded on $\overline D_r\times [0,\infty)$ by Lemma \ref{lem:dist-flow-maps-K}, we finally obtain, by \eqref{eq:E-second}, \eqref{eq:E-second2} and \eqref{eqn:prod-rule}, that $E''$ is bounded on $[0,\infty)$.
\end{proof} 
 
We can now apply Barbalat's lemma:
\begin{prop}
\label{prop:Barbalat}
	Let $K$ satisfy \ref{hyp:K} and $r<\pi/2-\e$. Suppose that $g'\geq0$ and $g'$ is continuously differentiable on $[0,4r^2]$. Let $\rho_0\in\P(\S)$ be such that $\supp(\rho_0)\subset \overline D_r$, and consider $\rho\in\Cont([0,\infty),\P(\S))$ the global weak solution to \eqref{eqn:model} starting from $\rho_0$ from Proposition \ref{prop:inv-cont}. Then
\[
\lim_{t\rightarrow\infty}\int_\S |\V(x, t)|^2 \d\rho_t(x) = 0.
\]
\end{prop}
\begin{proof}
	By Lemmas \ref{lem:E-lim} and \ref{lem:d2E_dt2} we know that $E(t)$ has a finite limit as $t \to \infty$ and $t\mapsto E''(t)$ is bounded on $[0,\infty)$. From Barbalat's lemma we then conclude that $E'(t) \to 0$ as $t \to \infty$, which by \eqref{eqn:dEdt-cont} leads to the desired result.
\end{proof}

From H\"{o}lder's inequality, an immediate consequence of this result is:
\be\label{C-1}
	\lim_{t\rightarrow0} \int_\S \V(x, t) \d \rho_t(x) = 0,
\ee
where $\rho$ and $\V$ are as in Proposition \ref{prop:Barbalat}.

We now want to conclude the asymptotic limit for the continuum model. For the considerations that follow it is convenient to use the following notation:
\[
G(x, y) = 2g'(d(x, y)^2)\frac{d(x, y)}{\sin(d(x, y))}, \qquad \mbox{for all $x,y\in\S$, $x \neq y$}.
\]
We also define $G(x,x) = 2 g'(0)$ for all $x \in \S$, by taking the limit $y \to x$ in the above.

Throughout the rest of this section we will make use of the following assumptions on $G$:
\begin{equation}\label{eqn:G-hyp1}
	G(x, y)\geq \c, \qquad \text{for all $x,y\in \S$, for some $\c>0$},
\end{equation}   
and 
\begin{equation} \label{eqn:G-hyp2}
	G(x_1, y_1)\geq G(x_2, y_2), \qquad \text{for all $x_1,y_1,x_2,y_2\in\S$ such that $d(x_1, y_1)\geq d(x_2, y_2)$}.
\end{equation}
Note that by \eqref{eqn:log-sphere}, given a global weak solution $\rho$ to \eqref{eqn:model} we have 
\begin{equation} \label{eqn:vG}
	\V(x,t) = - \int_{\S} G(x,y) (y-(x \cdot y) x) \d\rho_t(y), \qquad \mbox{for all $(x,t)\in\S\times [0,\infty)$}.
\end{equation}
We set an additional notation and define $\h\: \S\times [0,\infty) \to \R^d$ as:
\begin{equation} \label{eqn:defn-C}
	\h(x, t) = \int_{\S} G(x, y) y \d\rho_t(y), \qquad \mbox{for all $x,y \in \S\times [0,\infty)$},
\end{equation}
which enables us to rewrite \eqref{eqn:vG} further as
\begin{equation} \label{eqn:vC}
	\V(x,t) = -\h(x,t)+(\h(x,t) \cdot x) \, x.
\end{equation}

\begin{rem}
\label{rem:gG}
For convenience, we work with the assumptions \eqref{eqn:G-hyp1} and \eqref{eqn:G-hyp2} on $G$. We note, however, that in terms of the interaction function $g$, for \eqref{eqn:G-hyp1} and \eqref{eqn:G-hyp2} to be satisfied it is sufficient to assume that $g' \geq \c/2$ and $g'$ is nondecreasing.
\end{rem}

We present now some important technical lemmas which will be needed to prove our main consensus result given in Theorem \ref{thm:consensus-cont}. 
\begin{lem}\label{L-3}
	Let $\rho\in\P(\S)$ be such that $\supp(\rho)\subset \overline D_r$ for some $r<\pi/2-\e$, and assume that $G$ satisfies \eqref{eqn:G-hyp1}. Write $c(x) = \int_\S G(x,y) y \d\rho(y)$ for all $x\in\S$. Then,
\[
	\abs{c(x)} \geq \c \cos r \quad \mbox{for all $x\in \S$},
\]
and
\[
	c(x) \cdot z  \geq \abs{c(x)} \cos 2r, \quad \text{ for any $x\in\S$ and $z\in \overline D_r$}.
\]
\end{lem}
\begin{proof} 
Since the support of $\rho$ lies on the closed geodesic disk $\overline D_r$, we have
\begin{equation}
\label{eqn:ydotN}
	y \cdot \North \geq \cos r, \qquad \text{ for all } y\in \supp(\rho).
\end{equation}
Hence, by \eqref{eqn:G-hyp1} we have, for all $x\in\S$,
\[
	\abs{c(x)} \geq  c(x) \cdot \North = \int_\S G(x, y) y \cdot \North \d\rho(y) \geq \c\cos r,
\]
proving the first inequality.

To prove the second inequality we fix $x\in\S$ and assume that $\int_\S G(x,y) y \d\rho(y) \neq 0$, otherwise the result is trivial. We note that the unit vector $c(x)/\abs{c(x)}$ lies on the closed geodesic disk $\overline D_r$. Indeed, from \eqref{eqn:ydotN} one gets:
\begin{align}
	\frac{c(x)}{\abs{c(x)}} \cdot \North =\frac{\int_\S G(x, y) y \cdot \North \d \rho(y)}{|\int_\S G(x, y)y \d\rho(y) |} \geq  \frac{\int_\S G(x, y) \cos r \d \rho(y)}{\int_\S G(x, y) \d\rho(y)}=\cos r.
\end{align}
Hence, the angle between $\North$ and $c(x)/\abs{c(x)}$ is smaller than or equal to $r$, and therefore $c(x)/\abs{c(x)}$ belongs to $\overline D_r$. Now, take any $z\in \overline D_r$. Since both $c(x)/\abs{c(x)}$ and $z$ belong to $\overline D_r$, the angle between these two vectors is smaller than or equal to $2r$, leading to the second inequality.
\end{proof}

\begin{lem}\label{L-5}
	Let $\rho\in\P(\S)$ be such that $\supp(\rho)\subset \overline D_r$ for some $r<\pi/2-\e$, and assume that $G$ satisfies \eqref{eqn:G-hyp2}. Then, for any $x_1,x_2 \in \S$, it holds that
\begin{equation}
\label{eqn:L5}
	\h(x_1) \cdot x_1+ \h(x_2) \cdot x_2 \leq \h(x_1) \cdot x_2 + \h(x_2) \cdot x_1,
\end{equation}
where $c$ is as in Lemma \ref{L-3}.
\end{lem}
\begin{proof}
Let $x_1,x_2\in\S$. By a direct calculation, 
\begin{align}
&\h(x_1) \cdot x_1+ \h(x_2) \cdot x_2 - \h(x_1) \cdot x_2 - \h(x_2) \cdot x_1 \nonumber \\
&\qquad =\int_\S G(x_1, y)(y \cdot x_1- y \cdot x_2) \d \rho(y) + \int_\S G(x_2, y)(y \cdot x_2- y\cdot x_1) \d\rho(y) \nonumber \\
&\qquad =\int_\S (G(x_1, y)-G(x_2, y)) (y \cdot x_1- y \cdot x_2) \d\rho(y). \label{eqn:ineq-C}
\end{align}
If $y\in\S$ is such that $d(x_1, y)\geq d(x_2, y)$, then
\[
G(x_1, y)-G(x_2, y)\geq 0 \quad \text{ and } \quad y \cdot x_1- y \cdot x_2 \leq0, 
\]
where we used \eqref{eqn:G-hyp2} and the fact that $d(x_1,y)= \operatorname{arccos}(x_1 \cdot y)$ (and similarly for $x_2$).
Also, if $y\in\S$ is such that $d(x_1, y)\leq d(x_2, y)$, then
\[
G(x_1, y)-G(x_2, y)\leq 0,\quad \text{ and } \quad y \cdot x_1 - y \cdot x_2 \geq0.
\]
We conclude that the product
\[
(G(x_1, y)-G(x_2, y)) (y \cdot x_1- y \cdot x_2) \leq 0 \quad \mbox{for all $y\in\S$}.
\]
By \eqref{eqn:ineq-C}, one then concludes:
\[
\h(x_1) \cdot x_1+ \h(x_2) \cdot x_2 - \h(x_1) \cdot x_2 + \h(x_2) \cdot x_1 \leq 0. \qedhere
\]
\end{proof}

We finally give a lemma involving the asymptotic behaviour of the map $\h(x,t)$ defined in \eqref{eqn:defn-C}.
\begin{lem}\label{L-4}
	Let $K$ satisfy \ref{hyp:K} and $r<\pi/2-\e$. Suppose that $g'\geq0$ and $g'$ is continuously differentiable on $[0,4r^2]$. Also assume that $G$ satisfies \eqref{eqn:G-hyp1} (see Remark \ref{rem:gG}). Let $\rho_0\in\P(\S)$ be such that $\supp(\rho_0)\subset \overline D_r$, and consider $\rho\in\Cont([0,\infty);\P(\S))$ the global weak solution to \eqref{eqn:model} starting from $\rho_0$ from Proposition \ref{prop:inv-cont}. Then
\[
\lim_{t\rightarrow\infty} \int_\S (|\h(x, t)|- \h(x, t) \cdot x ) \d\rho_t(x) = 0.
\]
\end{lem}
\begin{proof}
By Proposition \ref{prop:Barbalat} and \eqref{eqn:vC} we have
\begin{equation}
\label{eqn:lim-vC}
\lim_{t\rightarrow\infty} \int_\S (|\h(x, t)|^2- (\h(x, t) \cdot x)^2) \d \rho_t(x)=0.
\end{equation}
From the second inequality in Lemma \ref{L-3}, we infer, for all $x\in\S$,
\[
	|\h(x, t)|+ \h(x, t) \cdot x \geq |\h(x, t)|(1+\cos2r).
\]
Hence, also using the first inequality in Lemma \ref{L-3}, 
\begin{align*}
	& \int_\S (|\h(x, t)|^2- (\h(x, t) \cdot x)^2) \d \rho_t(x) \\
	& \qquad \qquad =\int_\S (|\h(x, t)|-\h(x, t) \cdot x)(|\h(x, t)|+ \h(x, t)\cdot x)\d \rho_t(x) \\
	&\qquad \qquad \geq \c\cos r(1+\cos2r)\int_\S  (|\h(x, t)|-\h(x, t) \cdot x)\ \d \rho_t(x) \geq 0.
\end{align*}
From the estimate above and \eqref{eqn:lim-vC} we conclude the proof.
\end{proof}

We now state and prove an important lemma towards our consensus result.
\begin{lem} \label{lem:consensus-cont}
	Let $K$ satisfy \ref{hyp:K} and $r<\pi/2-\e$. Suppose that $g'\geq0$ and $g'$ is continuously differentiable on $[0,4r^2]$. Also assume that $G$ satisfies \eqref{eqn:G-hyp1} and \eqref{eqn:G-hyp2} (see Remark \ref{rem:gG}). Let $\rho_0\in\P(\S)$ be such that $\supp(\rho_0)\subset \overline D_r$, and write $\rho\in\Cont([0,\infty);\P(\S))$ the global weak solution to \eqref{eqn:model} starting from $\rho_0$ from Proposition \ref{prop:inv-cont}. Then
\begin{align*}
\lim_{t\rightarrow\infty} \int_\S \int_\S (1-x_1 \cdot x_2) \d \rho_t(x_1) \d\rho_t(x_2)=0.
\end{align*}
\end{lem}
\begin{proof}
	At start we shall suppress all the dependences on $t\in[0,\infty)$ for clarity. For all $x_1,x_2 \in \S$, subtract $(x_1 \cdot x_2) (\h(x_1) \cdot x_1+ \h(x_2) \cdot x_2)$ on both sides of \eqref{eqn:L5} to get: 
\begin{align*}
&(1-x_1 \cdot x_2) (\h(x_1) \cdot x_1+ \h(x_2) \cdot x_2)\\
&\qquad  \leq \h(x_1) \cdot x_2- (\h(x_1) \cdot x_1)(x_1 \cdot x_2)+ \h(x_2) \cdot x_1- (\h(x_2) \cdot x_2) (x_1 \cdot x_2).
\end{align*}
	Integrating the above inequality and passing all terms to the right-hand side:
\begin{align*}
	&0 \leq \int_\S \int_\S \left( \h(x_1) \cdot x_2 - (\h(x_1) \cdot x_1) (x_1 \cdot x_2) \right) \d\rho(x_1)\d\rho(x_2) \\
	&\quad + \int_\S \int_\S \left( \h(x_2) \cdot x_1 - (\h(x_2) \cdot x_2)(x_1 \cdot x_2) \right) \d\rho(x_1)\d\rho(x_2) \\
	&\quad - \int_\S \int_\S (1-x_1 \cdot x_2)(\h(x_1) \cdot x_1) \d\rho(x_1)\d\rho(x_2) - \int_\S \int_\S (1-x_1 \cdot x_2)(\h(x_2) \cdot x_2) \d\rho(x_1) \d\rho(x_2).
\end{align*}
Now add $\int_\S\int_\S (1-x_1 \cdot x_2)(|\h(x_1)|+|\h(x_2)|) \d\rho(x_1)\d\rho(x_2)$ to both sides to get:
\begin{align*}
	&\int_\S\int_\S (1-x_1 \cdot x_2)(|\h(x_1)|+|\h(x_2)|) \d\rho(x_1)\d\rho(x_2)\\
	&\quad \leq \int_\S \int_\S \left( \h(x_1) \cdot x_2 - (\h(x_1) \cdot x_1) (x_1 \cdot x_2) \right) \d\rho(x_1)\d\rho(x_2) \quad (:=\mathcal{I}_1) \\
	&\qquad + \int_\S \int_\S \left( \h(x_2) \cdot x_1 - (\h(x_2) \cdot x_2)(x_1 \cdot x_2) \right) \d\rho(x_1)\d\rho(x_2) \quad (:=\mathcal{I}_2)\\
	&\qquad + \int_\S \int_\S (1-x_1 \cdot x_2)(|c(x_1)|-\h(x_1) \cdot x_1) \d\rho(x_1)\d\rho(x_2) \quad (:=\mathcal{I}_3)\\
	&\qquad + \int_\S \int_\S (1-x_1 \cdot x_2)(|c(x_2)| - \h(x_2) \cdot x_2) \d\rho(x_1) \d\rho(x_2) \quad (:=\mathcal{I}_4).
\end{align*}
From the first inequality in Lemma \ref{L-3}, we have:
\begin{equation*}
	0\leq2\c\cos r  \int_\S\int_\S (1-x_1 \cdot x_2)\d\rho(x_1)\d\rho(x_2) \leq \int_\S\int_\S(1-x_1 \cdot x_2)(|\h(x_1)|+|\h(x_2)|) \d\rho(x_1)\d\rho(x_2).
\end{equation*}
Combining the above inequalities we get:
\be \label{eqn:I1-4}
0 \leq 2\c\cos r\int_\S\int_\S(1-x_1 \cdot x_2)\d\rho(x_1)\d\rho(x_2) \leq \mathcal{I}_1 + \mathcal{I}_2 + \mathcal{I}_3 + \mathcal{I}_4.
\ee

We now show that the each term $\mathcal{I}_1$, $\mathcal{I}_2$, $\mathcal{I}_3$ and $\mathcal{I}_4$ converges to $0$ as $t \to \infty$. Indeed, by restoring
the dependence on t, we have:
\begin{align*}
	\mathcal{I}_1=\int_\S x_2 \cdot \left( \int_\S (\h(x_1,t)- (\h(x_1,t) \cdot x_1) x_1) \d\rho_t(x_1) \right) \d\rho_t(x_2).
\end{align*}
By \eqref{C-1} and \eqref{eqn:vC} we get $\lim_{t\rightarrow\infty}\mathcal{I}_1=0$, and by a similar argument $\lim_{t\rightarrow\infty}\mathcal{I}_2=0$. For $\mathcal{I}_3$ we estimate (note that $|1-x_1 \cdot x_2| \leq 2$ and $|\h(x_1)|- \h(x_1)\cdot x_1 \geq 0$):
\be
	\mathcal{I}_3 \leq 2\int_\S\int_\S (|\h(x_1,t)|- \h(x_1,t) \cdot x_1) \d\rho_t(x_1)\d\rho_t(x_2) =2\int_\S (|\h(x_1,t)|- \h(x_1,t) \cdot x_1)  \d \rho_t(x_1) .
\ee
Since $\mathcal{I}_3 \geq 0$ and by Lemma \ref{L-4} the right-hand side of the inequality above approaches $0$ at infinity, we infer $\lim_{t\rightarrow\infty}\mathcal{I}_3=0$. A similar argument yields $\lim_{t\to\infty}
\mathcal{I}_4 = 0$. Finally, by passing to the limit $t \to \infty$ in \eqref{eqn:I1-4} we obtain:
\begin{align*}
0\leq \lim_{t\rightarrow\infty}2\c\cos r \int_\S\int_\S(1-x_1 \cdot x_2)\d\rho(x_1)\d\rho(x_2) \leq 0,
\end{align*}
which leads to the desired result.
\end{proof}

We can finally prove the main result of this section:
\begin{thm}[Asymptotic consensus in the continuum model]\label{thm:consensus-cont}
	Let $K$ satisfy \ref{hyp:K} and $r<\pi/2-\e$. Suppose that $g'\geq0$ and $g'$ is continuously differentiable on $[0,4r^2]$. Also assume that $G$ satisfies \eqref{eqn:G-hyp1} and \eqref{eqn:G-hyp2}. Let $\rho_0\in\P(\S)$ be such that $\supp(\rho_0)\subset \overline D_r$, and consider $\rho\in\Cont([0,\infty);\P(\S))$ the global weak solution to \eqref{eqn:model} starting from $\rho_0$ from Proposition \ref{prop:inv-cont}. Then there exists $p\in \overline D_r$ such that $W_1(\rho_t,\delta_p) \to 0$ as $t\to\infty$.
\end{thm}
\begin{proof}
	By Proposition \ref{prop:inv-cont}, for all $t\in[0,\infty)$ we have that $\supp(\rho_t)$ is a subset of $\overline D_r$, which is compact, so that Prokhorov's theorem ensures the existence of $\rho_\infty\in\P(\S)$ such that $\supp(\rho_\infty) \subset \overline D_r$ and $(\rho_t)_{t\geq0}$ converges narrowly to $\rho_\infty$. By compactness of the sphere we further get $W_1(\rho_t,\rho_\infty) \to 0$ as $t\to\infty$.
	
	Let $\phi:\S\times\S \to\R$ denote the map $(x_1,x_2) \mapsto 1-x_1\cdot x_2$, which we observe is continuous and bounded. We also note that the family $(\rho_t \otimes \rho_t)_{t\geq0}$ of product measures narrowly converges to $\rho_\infty\otimes\rho_\infty$. By Lemma \ref{lem:consensus-cont} we then have
	\bes
		0 = \lim_{t\to\infty} \int_\S \int_\S \phi(x_1,x_2) \d\rho_t(x_1)\d\rho_t(x_2) = \int_\S \int_\S \phi(x_1,x_2) \d\rho_\infty(x_1)\d\rho_\infty(x_2).
	\ees
	Since $\phi\geq0$ we get that $\phi(x_1,x_2) = 0$ for $\rho_\infty\otimes\rho_\infty$-almost all $(x_1,x_2)\in\S\times\S$. Suppose, by contradiction, that there exist $x_1,x_2\in\supp(\rho_\infty)$ with $x_1\neq x_2$. Then, there exists $\delta>0$ so that $B_\delta(x_1) \cap B_\delta(x_2) = \emptyset$ and $(\rho_\infty\otimes\rho_\infty)(B_\delta(x_1)\times B_\delta(x_2))>0$. Furthermore, there exists $(x_1',x_2') \in B_\delta(x_1)\times B_\delta(x_2)$ such that $\phi(x_1',x_2') =0$, that is, $x_1'\cdot x_2' = 1$. Since $x_1'$ and $x_2'$ lie on the sphere, this implies that $x_1'=x_2'$, which contradicts $B_\delta(x_1) \cap B_\delta(x_2) = \emptyset$. We infer that $\supp(\rho_\infty)$ is a singleton, which concludes the proof.
\end{proof}


\subsection{Asymptotic consensus in the discrete model}
\label{subsect:consensus-disc}
We turn now to the asymptotic behaviour of solutions in the discrete model with purely attractive interaction potentials. First we want to note that the theory developed in Section \ref{subsect:consensus-cont} (e.g., Theorem \ref{thm:consensus-cont}) considers weak measure-valued solutions, and in particular it applies to the discrete case as well. Nevertheless,  we prove below a consensus result for the discrete model that assumes weaker assumptions on the interaction potential.

Fix an integer $\N\geq2$ and, without loss of generality, consider $\N$ particles of identical masses $1/\N$ that evolve on $\S$ according to the discrete model \eqref{eqn:model-disc}, which then reads:
\be\label{eqn:model-discrete}
	\begin{cases} x_i'(t) = - \displaystyle \frac1\N  \sum_{j=1}^{\N} \nabla_{\Sph^\dim}K_{x_j(t)}(x_i(t)),\\ x_i(0) = x_i^0. \end{cases}
\ee
In analogy with the continuum model, we remark that the discrete model \eqref{eqn:model-discrete} is a gradient flow with respect to the discrete energy $E_n\: \S^\N \to \R$ given by:
\begin{equation}
\label{eqn:energy-discrete}
	E_n(x_1,\dots,x_\N)=\frac{1}{\N^2}\sum_{1\leq i < j\leq \N}K(x_i, x_j), \qquad\mbox{for all $(x_1,\dots,x_\N)\in\S^\N$}.
\end{equation}
Indeed, one can reformulate the first line in \eqref{eqn:model-discrete} as
\begin{equation}
\label{eqn:grad-flow}
	x_i'(t)=-\N\nabla_{\Sph^\dim}^iE_n(x_1(t),\dots,x_\N(t)),
\end{equation}
where $\nabla_{\Sph^\dim}^i$ stands for the manifold gradient with respect to the $i$th variable. This energy will play an important role in the considerations below.

We present a technical lemma first. 
\begin{lem}
\label{lem:technical}
	Let $x_1, \dots, x_{\N}\in D_{\pi/4}$ be such that 
\[
 \Delta:=\max_{1\leq i,j\leq n}d(x_i,x_j) >0.
\]
By reindexing if necessary, assume that $d(x_1,x_2) = \Delta$. Then,
\[
\log_{x_1}x_2 \cdot  \log_{x_1}x_j \geq 0, \qquad \mbox{for all $j\in\{1,\dots,\N\}$}.
\]
\end{lem}
\begin{proof}
Consider the closed disk $\overline{D}_\Delta(x_2)$ centred at $x_2$ with radius $\Delta$. Then, by definition of $\Delta$ and the fact that $d(x_1,x_2) = \Delta$, we have $x_j \in \overline{D}_\Delta(x_2)$ for all $j\in\{1,\dots,\N\}$. 

If $\N=2$, then the result is trivial; suppose that $\N\geq3$. For $j \in\{3,\dots,\N\}$ fixed, consider the minimizing geodesic between $x_1$ and $x_j$. Parametrize this geodesic by $x(t)$, with $x(0)=x_1$ and $x'(0)=\log_{x_1}x_j$; in particular, $x(1) = x_j$ and $x(t)\neq x_2$ for all $t\in[0,1]$. Then, by the chain rule and \eqref{eqn:gradd} we find
\be\label{eqn:dy2}
	\frac{\der}{\der t} d(x(t), x_2)^2 = \grad_{\Sph^\dim}d_{x_2}^2(x(t)) \cdot x'(t) =-2 \log_{x(t)}x_2 \cdot x'(t).
\ee
Note that by the geodesic convexity of $\overline{D}_\Delta(x_2)$, we have $x(t) \in \overline{D}_\Delta(x_2)$ and $d(x(t),x_2) \leq d(x_1,x_2)$ for all $t\in[0,1]$. Hence the map $t\mapsto d(x(t),x_2)^2$ is nonincreasing at $t=0$, and by setting $t=0$ in \eqref{eqn:dy2} we find
\begin{equation*}
0 \geq \left.\frac{\der}{\der t}\right|_{t=0} d(x(t), x_2)^2 =-2 \log_{x_1}x_2 \cdot \log_{x_1}x_j,
\end{equation*}
which yields the desired conclusion.
\end{proof}

The following theorem shows the asymptotic convergence towards a consensus/synchronized state for the intrinsic model on sphere.
\begin{thm}[Asymptotic consensus in discrete model]
\label{thm:consensus-disc}
	Let $K$ satisfy \ref{hyp:K} and $r< \pi/4-\e$. Assume that $g'$ has continuous derivative on $[0,4r^2]$ and satisfies $g'(s) \geq \c s^\alpha$ for all $s \in [0,4r^2]$, for some $\c>0$ and $\alpha \geq 0$. Let furthermore  $(x_i^0)_{i=1}^n \subset \overline D_r$. Then the unique global solution $(x_i)_{i=1}^n$ to \eqref{eqn:model-discrete} from Proposition \ref{prop:inv-disc} is such that $d(x_i(t), x_j(t))\to 0$ as $t \to \infty$ for every $i,j\in\{1,\dots \N\}$.
\end{thm}
\begin{proof}
We proceed in two steps.
\smallskip

{\em Step 1.} By abuse of notation, denote $t\mapsto E_n(t)$ the map such that $E_n(t) = E_n(x_1(t),\dots,x_n(t))$ for all $t\in[0,\infty)$, where we recall that the discrete energy $E_n$ is given by \eqref{eqn:energy-discrete}. Writing $\rho_0 = \tfrac1\N\sum_{i=1}^n \delta_{x_i^0}$, we have $\supp(\rho_0) \subset \overline D_r$, and Proposition \ref{prop:inv-cont} gives us the existence of a unique global weak solution $\rho$ to \eqref{eqn:model} starting from $\rho_0$. From Lemma \ref{lem:atomic-sol}, the unique global weak solution $\rho^n$ starting from $\rho_0$ from Proposition \ref{prop:inv-cont} reads:
\bes
	\rho_t^n = \frac1\N \sum_{i=1}^n \delta_{x_i(t)},  \qquad \mbox{for all $t\in[0,\infty)$}.
\ees
Noting that the discrete energy $E_n(t) = E[\rho_t^n]$, where $E$ is the continuum energy \eqref{eqn:energy-cont}, we obtain from Lemmas \ref{lem:E-lim} and \ref{lem:d2E_dt2} that $E_n(t) \to E_\infty$ as $t\to\infty$ for some $E_\infty\in\R$ and the map $t\mapsto E_n''(t)$ is bounded on $[0,\infty)$.

By applying Barbalat's lemma to $t\mapsto E_n(t)$, we then get
\[
	E_n'(t) \rightarrow 0, \qquad \text{ as  $t \to \infty$}.
\]
Using \eqref{eqn:grad-flow}, we compute, for all $t\in[0,\infty)$,
\begin{equation}
	E_n'(t)=\sum_{i=1}^{\N} \nabla_{\Sph^\dim}^i E_n(x_1(t),\dots,x_\N(t)) \cdot x'_i(t) = -\frac{1}{\N}\sum_{i=1}^{\N} |x'_i(t)|^2\leq0,
\end{equation}
which then implies that
\begin{equation}
\label{eqn:equil}
	x_i'(t)\to 0,  \qquad \text{ as $t \to \infty$ for all $i\in\{1,\dots,\N\}$}.
\end{equation}
\smallskip

{\em Step 2.} Recall from Proposition \ref{prop:inv-disc} that $x_i(t)\in \overline D_r$ for all $t\in[0,\infty)$ and $i\in\{1,\dots,\N\}$; in particular, any particles stay within distance $2r$ at all times. Let $\Delta\: [0,\infty) \to [0,\infty)$ be given by
\[
	\Delta(t) = \max_{1\leq i,j\leq n}d(x_i(t),x_j(t)) \quad \mbox{for all $t\in[0,\infty)$}.
\]
We want to show that $\Delta(t) \to 0$ as $t \to \infty$, which will conclude the proof. We will use Lemma \ref{lem:technical}. 
Reindexing particles at all times if necessary, assume that 
\[
d(x_1(t),x_2(t)) = \Delta(t), \qquad \mbox{for all $t\in[0,\infty)$}.
\]
Taking the inner product with $\log_{x_1(t)}x_2(t)$ on both sides of \eqref{eqn:model-discrete} for particle $i=1$, we get
\begin{align}
\label{est:dp}
	x_1' \cdot \log_{x_1}x_2 &= \frac{1}{\N}\sum_{j=1}^{\N} 2g'(d(x_1, x_j)^2)\log_{x_1}x_j \cdot \log_{x_1}x_2 \nonumber \\
& \geq \frac{2}{\N} g'(d(x_1, x_2)^2) | \log_{x_1}x_2 |^2, 
\end{align}
where we dropped the dependence on $t$ for simplicity, and where for the inequality on the second line we used Lemma \ref{lem:technical} to bound from below a sum of nonnegative terms by the second term.

Using the Cauchy--Schwarz inequality $|x_1' \cdot \log_{x_1}x_2 | \leq |x_1'| |\log_{x_1}x_2|$ and the fact that $|\log_{x_1}x_2| = d(x_1,x_2)$, from \eqref{est:dp} we find
\[
\frac{2}{\N} g'(d(x_1, x_2)^2) d(x_1,x_2) \leq |x_1'|.
\]
Finally, using the bound assumption on $g'$ we get, for all $t\in[0,\infty)$,
\[
\frac{2}{\N} \c d(x_1,x_2)^{1+2\alpha} \leq |x_1'|,
\]
where $1+2\alpha >0$. And since by \eqref{eqn:equil}, $x_1'(t)$ approaches $0$ as $t \to \infty$, so does $d(x_1(t),x_2(t))$. Hence
\[
	\Delta(t) \to 0,  \qquad \text{ as }  t \to \infty. \qedhere
\]
\end{proof}

\paragraph{Examples.} We discuss here some examples of interaction potentials that satisfy the assumptions in Theorem \ref{thm:consensus-disc}. 
\begin{enumerate}
\item {\em Power-law potentials.} The quadratic potential
\[
K(x, y)=d(x, y)^{2}, \quad \text { for } g(s)=s,
\]
satisfies $g'(s) \geq \c s^\alpha$ for all $s\in[0,\infty)$ with $\c=1$ and $\alpha = 0$, and $g$ is furthermore of class $C^2$. More generally, for $q \geq 2$,
\[
K(x, y)=d(x, y)^{2q}, \quad \text { for } g(s)=s^q,
\]
satisfies $g'(s) \geq \c s^\alpha$ for all $s\in[0,\infty)$ with $\c=q$ and $\alpha = q-1$, and is of class $C^2$.

Interaction potentials in power-law form have been one of the main types of potentials investigated in the aggregation literature  \cite{Balague_etalARMA, BaCaLaRa2013, ChFeTo2015, FeHuKo11, FeHu13}. Despite their simplicity, it was shown that they can capture a wide variety of  ``swarm" behaviours, such as aggregations on disks, annuli, rings, delta concentrations, and others, for both the model with extrinsic interactions\cite{KoSuUmBe2011}, as well as for the intrinsic model investigated in this paper \cite{FeZh2019}.

\item {\em Potential in Lohe sphere model.} The potential
\[
K(x, y)=2\sin^2\left(\frac{{d(x, y)}}{2}\right), \qquad \text{ for } g(s)=2 \sin^2\left(\frac{\sqrt{s}}{2}\right),
\]
corresponds to the Lohe sphere model studied in various recent papers \cite{HaKoRy2018,HaKiLeNo2019}. Indeed, the discrete Lohe model on the unit sphere reads:
\begin{align*}
x_i'=\Omega_i x_i+\frac{\kappa}{\N}\sum_{k=1}^\N(x_k- (x_k \cdot x_i)  x_i), \qquad i\in\{1,\dots,\N\},
\end{align*}
where $\Omega_i$ is a natural frequency matrix and $\kappa$ is a coupling strength. As done previously, all particles $x_i(t)\in \Sph^{\dim}$ for all $t\in[0,\infty)$ are considered as vectors in $\R^{\dim+1}$. Given that on the unit sphere, $d(x, y)=\arccos (x \cdot y)$ for all $x,y\in\Sph^\dim$, from the identity $\cos\theta=1-2\sin^2(\theta/2)$, one can write the potential as:
\[
K(x, y)=1-\cos d(x, y)=1- x \cdot y=\tfrac{1}{2}|x-y|^2.
\]
Therefore, $K$ can also be regarded as a quadratic potential with respect to the Euclidean distance in the ambient space $\R^{\dim+1}$. The Euclidean gradient of $K$ is given by
\[
-\nabla K_y(x)=y, \qquad \mbox{for all $x,y\in\S$},
\]
and projecting it onto the tangent plane to the sphere one gets the manifold gradient of $K$:
\[
-\nabla_{\Sph^\dim}K_y(x)=y- (x \cdot y) x, \qquad \mbox{for all $x,y\in\S$},
\]
which is the coupling term in the Lohe sphere model. Compute
\[
g'(s)=\frac{1}{2} \frac{\sin\frac{\sqrt{s}}{2}}{\frac{\sqrt{s}}{2}}\cos\frac{\sqrt{s}}{2}, \qquad \mbox{for all $s\in[0,\infty)$}.
\]
Take an initial particle configuration of particles in a geodesic disk $D_r$, with $r<\pi/2-\e$. The function $g'$ verifies
\[
g'(s)\geq \frac{\cos r}{2}, \qquad \mbox{for all $s \in [0,4r^2]$},
\]
so it satisfies the bound condition of Theorem \ref{thm:consensus-disc} with $\c= \cos(r) /2$ and $\alpha =0$. Finally, it can also be checked that $g$ is of class $C^2$.
\end{enumerate}

\begin{rem}
Among the examples above, we note that only the quadratic potential and the Lohe sphere potential satisfy the assumptions of Theorem \ref{thm:consensus-cont}, our continuum result. Indeed, for the quadratic potential the function $g$ satisfies the sufficient conditions given in Remark \ref{rem:gG}, while for the Lohe model, a direct calculation shows $G(x,y)=1$ for all $x,y\in\S$. Higher-order power-law potentials, however, do not satisfy \eqref{eqn:G-hyp1} as $G$ is not bounded below by a positive constant in this case. This illustrates the fact that our discrete result (Theorem \ref{thm:consensus-disc}) holds for a wider class of potentials than our continuum counterpart (Theorem \ref{thm:consensus-cont}).
\end{rem}


\section{Intrinsic aggregation model and consensus on other manifolds}
\label{sect:other}
In this section we consider the intrinsic aggregation model and its asymptotic behaviour on other manifolds, in particular on a hypercylinder.

\subsection{Intrinsic aggregation model on cylinder}
\label{subsect:cylinder}
We show that results similar to those in Section \ref{sect:sphere} can be obtained for a cylinder in $\R^3$, or, more generally, for a hypercylinder in arbitrary dimension. Here, by a hypercylinder in $\R^{\dim+1}$ we mean the product manifold of a circle (endowed with the induced metric from $\R^2$) with $\R^{\dim-1}$, canonically embedded in $\R^{\dim+1}$. Similarly as in the case of the sphere, this embedding in $\R^{k+1}$ allows us to treat points and tangent vectors of $\Cyl^\dim$ as vectors in $\R^{k+1}$. For simplicity, we present the calculations for the cylinder in $\R^3$; extending the considerations to a hypercylinder would be immediate.

Consider the wrapping parametrization $(\cos x, \sin x, z)$ with $(x,z) \in [0, 2\pi) \times \mathbb{R}$ of a cylinder in $\R^3$.  Similarly to the sphere, we restrict our study to a subset of the cylinder where no two points are in the cut locus of each other. Note that for a point on the cylinder, its cut locus consists in the line on the cylinder opposite to it. Specifically, consider the subset $\C$ of the cylinder that corresponds, under the wrapping map, to the band $(0, \pi-\e) \times \mathbb{R}$, where $0<\e<\pi$ is fixed, that is,
\[
\C = \{(\cos x, \sin x, z) \mid  x \in (0, \pi-\e),\, z\in \mathbb{R} \}.
\]
This subset is contained in an open half-cylinder and hence, the cut locus of each point in it lies outside the set; see \eqref{eqn:setS} to compare with the sphere case. 

The wrapping parametrization is an isometry between the $xz$-plane and the cylinder. The metric is the identity matrix, and so is its inverse (as is for the Euclidean plane). Take two generic points $(x,z)$ and $(\barx,\barz)$ on the band $(0, \pi-\e) \times \mathbb{R}$, corresponding to points $P$ and $Q$ on $\C$.  The distance on the cylinder between $P$ and $Q$ is the distance between the two points on the plane:
\[
	d(P,Q) = \left( (x-\barx)^2 + (z-\barz)^2\right)^{1/2}.
\]

A well-posedness result for the intrinsic model \eqref{eqn:model} on the cylinder would follow as in Section \ref{sect:sphere}, provided analogues to Lemmas \ref{lem:dist-flow-maps}, \ref{lem:dist-flow-maps-K} and \ref{lem:grad-lip-2} are established for the cylinder (Lemma \ref{lem:Lipschitz-initial} following similarly). We sketch briefly the arguments leading to such analogous lemmas. Using the formula for surface gradient in coordinates one can compute, for points $P=(x,z)$ and $Q=(\bar x,\bar z)$ in the cylinder:
\begin{align*}
	\nabla_{\Cyl^2} d_Q^2(P) &= \frac{\partial}{\partial x} d^2(P,Q) \, \bee_x +  \frac{\partial}{\partial z} d^2(P,Q) \, \bee_z \\
&= 2 (x - \barx) \bee_x + 2 (z-\barz) \bee_z,
\end{align*}
where $\bee_x$ and $\bee_z$ are the tangent vectors along coordinate lines at $P$, given by:
\[
\bee_x = (-\sin x, \cos x, 0), \quad \text{ and } \quad \bee_z= (0,0,1).
\]
Also, the logarithm map in coordinates is given by:
\[
\log_P Q = (\barx-x) \bee_x + (\barz-z) \bee_z.
\]

{\em Analogue of Lemma \ref{lem:dist-flow-maps}.} Consider two time-dependent vector fields $X$ and $Y$ on $\C$ and let $\Sigma\subset \C$. Let moreover $\Psi^t_X$ and $\Psi^t_Y$ be the flow maps defined on $\Sigma\times[0,\tau)$, for some $\tau>0$, generated by $(X,\Sigma)$ and $(Y,\Sigma)$. We also assume that $X$ bounded on $\C\times[0,\tau)$ and Lipschitz continuous with respect to its first variable on $\C\times[0,\tau)$ (i.e., it satisfies \eqref{eq:Lipschitz-X} on $\C\times\C\times[0,\tau)$ for some $L_X>0$).

Fix $p\in \Sigma$ and $t\in[0,\tau)$; we will be using $(x,z)$ as coordinates for $P=\PsiX(p)$ and $(\barx,\barz)$ for $Q=\PsiY(p)$. Suppose $P\neq Q$ and compute:
\begin{align*}
	\frac{\der}{\der t} d(P,Q) &= \nabla_{\Cyl^2} d_Q(P) \cdot X_t(P) + \nabla_{\Cyl^2} d_P(Q) \cdot Y_t(Q) \\
	&  = \frac{1}{d(P,Q)} \Bigl( \underbrace{((x - \barx) \bee_x + (z-\barz) \bee_z)}_{:=A} \cdot X_t(P) +  \underbrace{((\barx - x) \bee_{\barx} + (\barz-z) \bee_{\barz})}_{:=B} \cdot Y_t(Q)\Bigr).
\end{align*}
Note that $|A| = |B| = d(P,Q)$. Add and subtract $A\cdot X_t(Q)$ and $B \cdot  X_t(Q)$ to the term in between the large brackets in the right-hand side above. Then estimate this term as:
\begin{align*}
A \cdot X_t(P) +  B \cdot Y_t(Q) & = A \cdot X_t(P) - A \cdot X_t(Q) + A \cdot  X_t(Q) + B \cdot  X_t(Q) - B \cdot  X_t(Q) + B \cdot Y_t(Q) \\
& \leq |A| \underbrace{|X_t(P) - X_t(Q)|}_{\leq L_X \, d(P,Q)} + |A+B| \| X\|_{L^\infty(\C\times[0,\tau))} + |B| \|X-Y \|_{L^\infty(\C\times[0,\tau))}.
\end{align*}
In computing $A+B$ the $z$-terms cancel and we get
\begin{align*}
|A+B| &= |x-\barx| |(-\sin x + \sin \barx, \cos x - \cos \barx, 0)| \\
& \leq \sqrt{2}  |x-\barx|^2 \leq \sqrt{2} \, {d(P,Q)}^2.
\end{align*}

Putting the estimates together we find:
\[
	\frac{\der}{\der t} d(P,Q) \leq (L_X+ \sqrt{2} \| X\|_{L^\infty(\C\times[0,\tau))}) d(P,Q) + \|X-Y \|_{L^\infty(\C\times[0,\tau))},
\]
and Gronwall's lemma will yield a similar result as in Lemma  \ref{lem:dist-flow-maps}.


\smallskip
{\em Analogue of Lemma \ref{lem:dist-flow-maps-K}.} To get a similar Lipschitz property, we consider three generic points $P=(x,z)$, $Q=(\barx,\barz)$ and $R=(\barx,\barz)$, and estimate:
\begin{align*}
\left | \nabla_{\Cyl^2} d_R^2(P) - \nabla_{\Cyl^2} d_R^2(Q) \right| &= 2 |(x-\tx) \bee_x + (z-\tz) \bee_z  - (\barx-\tx) \bee_{\barx} + (\barz-\tz) \bee_{\barz} | \\ 
& \leq 2 (| (x-\barx) \bee_x| +|(\barx-\tx) (\bee_x - \bee_{\barx})| + |(z-\barz)\bee_z|), 
\end{align*}
where in the above we added and subtracted $(\barx-\tx) \bee_x$, used $\bee_z = \bee_{\barz}$, and the triangle inequality. Then use
\[
|\bee_x - \bee_{\barx}| = |(-\sin x + \sin \barx, \cos x - \cos \barx, 0)| \leq \sqrt{2} |x - \barx|,
\]
together with $|\barx-\tx| \leq \pi$ and $|\bee_x| = |\bee_z| = 1$ to get
\begin{align}
\label{eqn:gradd2-cyl}
\left | \nabla_{\Cyl^2} d_R^2(P) -  \nabla_{\Cyl^2} d_R^2(Q) \right| &\leq 2 ((1+ \sqrt{2} \pi) |x - \tx| + |z- \tz|) \nonumber \\
& \leq 2 \sqrt{2} (1+ \sqrt{2} \pi) d(P,Q).
\end{align}

For a potential that satisfies \ref{hyp:K}, one estimates:
\begin{align}
\label{eqn:gradK-cyl}
& |\nabla_{\Cyl^2} K_R (P) - \nabla_{\Cyl^2} K_R(Q)|= |g'(d(P,R)^2) \nabla_{\Cyl^2} \, d_R^2(P) - g'(d(Q,R)^2) \nabla_{\Cyl^2} \, d_R^2(Q)| \nonumber \\[3pt]
& \qquad \leq |g'(d(P,R)^2) - g'(d(Q,R)^2)| |\nabla_{\Cyl^2} \, d_R^2(P)| + |g'(d(Q,R)^2)| |\nabla_{\Cyl^2} \, d_R^2(P) -  \nabla_{\Cyl^2} \, d_R^2(Q)| \nonumber \\[3pt]
& \qquad \leq 2 L_{g'} |d(P,R)^2 - d(Q,R)^2| d(P,R) + C_{g'}  |\nabla_{\Cyl^2} \, d_R^2(P) -  \nabla_{\Cyl^2} \, d_R^2(Q)| \nonumber \\
& \qquad \leq (2 L_{g'} (d(P,R) + d(Q,R)) d(P,R) +  C_{g'} 2 \sqrt{2} (1+ \sqrt{2} \pi)) d(P,Q),
\end{align}
where we added and subtracted $g'(d(Q,R)^2) \nabla_{\Cyl^2} \, d_R^2(P)$ on the first line and then used the triangle inequality, we used the bound and Lipschitz constant of $g'$ for the second inequality, and triangle inequality $|d(P,R)- d(Q,R)| \leq d(P,Q)$ and \eqref{eqn:gradd2-cyl} for the last inequality.

There is an important word of caution here, as the cylinder is unbounded and the bounds and Lipschitz constants of $g'$ need to be taken on compact sets. We deal with this issue similarly to how Ca{\~n}izo {et al. }\cite{CanizoCarrilloRosado2011} have dealt with the unboundedness of the Euclidean space. Namely, we consider estimates such as \eqref{eqn:gradK-cyl} only for points in an a priori fixed compact subset of $\C$, say of diameter $\diam$, in which case, for preciseness, one has to indicate the dependence of the constants on this diameter (i.e., $L_{g'}(\diam)$ and $C_{g'}(\diam)$). We also note that, for simplicity of notation, we have not indicated the dependence on $\e$ of the various constants, as we did for the sphere.

With this clarification, from \eqref{eqn:gradK-cyl} one can find the following Lipschitz estimate on a compact set of diameter $\diam$:
\begin{equation}
\label{eqn:gradKs-cyl}
	|\nabla_{\Cyl^2} K_R (P) - \nabla_{\Cyl^2} K_R(Q)| \leq L_{\diam} d(P,Q),
\end{equation}
where 
\[
L_{\diam} = 4 L_{g'}(\diam) \diam^2 +  C_{g'}(\diam) 2\sqrt{2} (1+ \sqrt{2} \pi).
\]
Then, similarly to the sphere case, one can use \eqref{eqn:gradKs-cyl} and the analogue of \eqref{eqn:X-diff} for the cylinder to establish a Lipschitz estimate as in Lemma \ref{lem:dist-flow-maps-K} for the vector field $\V[\rho]$ on $\C$, for $\rho\in\Cont([0,T);\P_\infty(\C))$ such that $\rho_t$ is supported within a compact subset of $\C$ of diameter $\Delta$ for all $t\in[0,T)$, where the Lipschitz constant is given by $L_{\diam}$ above.

On the other hand, the boundedness of $\V[\rho]$ is immediate. Indeed, for $\rho \in \Cont([0,T);\P_\infty(\C))$ such that $\rho_t$ is supported within a compact subset of $\C$ of diameter $\Delta$ for all $t\in[0,T)$, one has, for all $(x,t) \in \C\times[0,T)$:
\begin{align}
\label{eqn:boundX-cyl}
	|\V[\rho](x,t)| & \leq \int_{\C} |g'(d(x,y)^2) \nabla_{\Cyl^2} d_y^2(x)| \d\rho_t(y) \nonumber \\
	&\leq 2 \Delta C_{g'}(\Delta).
\end{align}


\smallskip

{\em Analogue of Lemma \ref{lem:grad-lip-2}.} Finally, we show how one can get a Lipschtiz condition of type \eqref{eqn:X-Yest} on the cylinder. For three points $P=(x,z)$, $Q=(\barx,\barz)$ and $R=(\tx,\tz)$, one finds:
\begin{align}
\label{eqn:gradd2-cyl2}
\left | \nabla_{\Cyl^2} d_Q^2(P) -  \nabla_{\Cyl^2} d_R^2(P) \right| &= 2|(\barx - \tx) \bee_x + (\barz - \tz) \bee_z| \nonumber \\
& = 2 \left( (\barx - \tx)^2 + (\barz - \tz)^2\right)^{1/2} \nonumber \\
&= 2 \, d(Q,R).
\end{align}
Then, for a potential $K$ that satisfies \ref{hyp:K}, we estimate:
\begin{align}
\label{eqn:gradK-cyl2}
& |\nabla_{\Cyl^2} K_Q(P) - \nabla_{\Cyl^2} K_R(P)|= |g'(d(P,Q)^2) \nabla_{\Cyl^2} d_Q^2(P) - g'(d(P,R)^2) \nabla_{\Cyl^2} d_R^2(P)| \nonumber \\[3pt]
&\qquad \leq  |g'(d(P,Q)^2) - g'(d(P,R)^2)| |\nabla_{\Cyl^2} d_Q^2(P)| + | g'(d(P,R)^2)| |\nabla_{\Cyl^2} d_Q^2(P) - \nabla_{\Cyl^2} d_R^2(P)| \nonumber \\[3pt]
&\qquad \leq (4 L_{g'}(\Delta) \Delta^2 + 2 C_{g'}(\Delta)) d(Q,R),
\end{align}
where we added and subtracted $g'(d(P,R)^2) \nabla_{\Cyl^2} d_Q^2(P)$ on the first line and then used the triangle inequality, and for the second inequality we used the bounds and Lipschitz constant of $g'$, the triangle inequality $|d(P,Q) - d(P,R)| \leq d(Q,R)$, equation \eqref{eqn:gradd2-cyl2}, and the a priori assumption that the three points lie on a set of diameter $\Delta$.

Similar to the proof for the sphere in Lemma \ref{lem:grad-lip-2}, estimate \eqref{eqn:gradK-cyl2} leads to a Lipschitz condition like \eqref{eqn:X-Yest} for $\rho,\sigma \in\Cont([0,T);\P_\infty(\C))$ such that $\rho_t$ and $\sigma_t$ are supported within a compact subset of $\C$ of diameter $\Delta$, with the Lipschitz constant given by:
\[
\Lip_\Delta = 4 L_{g'}(\Delta) \Delta^2 + 2 C_{g'}(\Delta).
\]

The considerations above lead to the following well-posedness result on the cylinder.
\begin{thm}[Well-posedness on open half-cylinder]\label{thm:well-posedness-cyl}
	Suppose that $K$ satisfies \ref{hyp:K} and let $\rho_0 \in \P_\infty(\C)$. Then, there exist a time $T>0$, a compact set $\supp(\rho_0)\subset Q\subset \C$, and a unique weak solution to \eqref{eqn:model} among all curves in $\Cont([0,T);\P(Q))$ starting from $\rho_0$, where $\P(Q)$ denotes the set of probability measures which are supported within $Q$. 
	\end{thm}
\begin{proof}
The proof is very similar to that of Theorem \ref{thm:well-posedness} and we just sketch it here. Because of the unboundedness of $\C$, the proof slightly differs from the case of the sphere; as already mentioned, we deal with this issue by considering only solutions which stay supported within an a priori fixed compact set. More specifically, in this proof take $z_m<z_M$ such that $\text{supp}(\rho_0)$ lies within the cylindrical band between $z=z_m$ and $z=z_M$, and write $Q_\Delta$ the compact cylindrical band between $z=2 z_m - z_M$ and $z= 2z_M-z_m$, whose diameter we denote by $\Delta$. 

The idea is to consider a map $\Gamma$ on $\Cont([0,\tau);\P(Q_\Delta))$, analogously defined as in \eqref{eqn:Gamma}, where $\tau>0$ is the maximal time so that $\Gamma$ is well-defined, and to show that, if restricted to some time interval $[0,T)$ with $T\leq\tau$ small enough, $\Gamma$ is a map from $\Cont([0,T);\P(Q_\Delta))$ into itself and a contraction.

We first show that $\supp(\Gamma(\sigma)(t)) \subset Q_\Delta$ for all $t\in[0,T)$, for some $T\leq \tau$ to be chosen later, and for all $\sigma \in \Cont([0,\tau);\P(Q_\Delta))$. Take such $\sigma$ and then, given the bound \eqref{eqn:boundX-cyl} on the velocity field $\V[\sigma]$ and the bounds on $\supp(\rho_0)$, note that $\text{supp}(\Gamma(\sigma)(t))$ lies within the cylindrical band between $z=2 z_m - z_M$ and $z= 2z_M-z_m$, i.e., within $Q_\Delta$, provided $t<(z_M-z_m)/(2 \Delta C_{g'}(\Delta)) =: T$. Similarly as in the proof of Theorem \ref{thm:well-posedness}, one then finds that $\Gamma$ defines a map from the metric space $(\Cont([0,T);\P(Q_\Delta)),\bd_1)$ into itself. 

The rest of the proof, that is, showing that $\Gamma$ is a contraction, follows exactly as for Theorem \ref{thm:well-posedness} (by eventually restricting $T$ further) using the analogues of Lemmas  \ref{lem:dist-flow-maps}, \ref{lem:dist-flow-maps-K}, \ref{lem:grad-lip-2} for the cylinder, as established above. We leave the details to the reader.
\end{proof}

\begin{rem}
	Thanks to the considerations above, the results for the sphere of Section \ref{subsect:stability} hold analogously for the cylinder; in particular, stability and mean-field limit hold true on the cylinder.
\end{rem}

\subsection{Consensus on product manifolds}
\label{subsect:product}

In this subsection we consider the intrinsic aggregation model on product manifolds. Specifically, given two smooth, complete and connected Riemannian manifolds $(M_1,g_1)$, $(M_2,g_2)$, we consider $M=M_1\times M_2$ with the product metric $g_1+g_2$ \cite{Lee1997}. The goal is to infer the formation of consensus on the product manifold $M$ from aggregation phenomena known on $M_1$ and $M_2$. We denote by $\U_1$ and $\U_2$ two generic open, geodesically convex subsets of $M_1$ and $M_2$, respectively, and set $\U=\U_1\times\U_2$.

The minimizing geodesic $\gamma$ connecting points $(x,y),(\bar{x},\bar{y})\in\U$ can be expressed as:
\[
\gamma(t)=(\gamma_1(t), \gamma_2(t)), \qquad \text{for } t\in[0, 1],
\]
where $\gamma_1$ and $\gamma_2$ are the minimizing geodesics connecting $x$ and $\bar{x}$ on $\U_1$, and $y$ and $\bar{y}$ on $\U_2$, respectively. We consider the product distance between the two points on $M$ to be given by:
\begin{equation}
\label{eq:prod-dist}
d((x,y), (\bar{x},\bar{y}))=\sqrt{d_1(x,\bar x)^2+d_2(y, \bar{y})^2},
\end{equation}
where $d_1$ and $d_2$ are the Riemannian distances on $M_1$ and $M_2$, respectively.  Finally, particularly important for the considerations of this section, from the definition of the product manifold the Riemannian logarithm on $\U$ is given by:
\begin{equation}
\label{eqn:logM}
\log_{(x, y)}(\bar{x}, \bar{y})=(\log_{x} \bar{x}, \log_{y} \bar{y}).
\end{equation}

\begin{ex}
\label{ex:prod-manifolds}
We give below a few examples of common product manifolds.
\begin{enumerate}
\item {\em Euclidean space} $\mathbb{R}^{\dim+\tilde{\dim}}=\mathbb{R}^\dim \times \mathbb{R}^{\tilde{\dim}}$.
\item {\em Cylinder} $\mathbb{S}^1\times\R$, where $\mathbb{S}^1$ represents the unit circle with induced metric from $\R^2$.
\item {\em Flat torus} $\mathbb{S}^1\times \mathbb{S}^1$, considered as a subset of $\R^4$, where $\mathbb{S}^1$ has the  induced metric from $\R^2$.
\end{enumerate}
\end{ex}

Consider now our intrinsic aggregation model on a product manifold $M$ with a (purely attractive) quadratic potential given by:
\begin{equation}
\label{eqn:K-quad}
K(z,\bar{z}) = \frac{1}{2} d^2(z,\bar{z}), \qquad \mbox{for all $z,\bar z\in M$},
\end{equation}
i.e., an interaction potential in the form \eqref{eqn:K-gen}, with $g(s)=s/2$. For this potential,
\[
	\nabla_{\M} K_{\bar z}(z) = -\log_z \bar{z},  \qquad \mbox{for all $z,\bar z\in \U$}.
\]

We first show that with this interaction potential, solutions to the aggregation model on the product manifold $\M$ can be obtained from solutions on its components $M_1$ and $M_2$. 
\begin{prop}[Well-posedness in continuum product model]\label{prop:well-posedness-prod}
	Let $K$ be as in \eqref{eqn:K-quad}, and suppose that there exist unique solutions $\rho^1\in\Cont([0,T);\P_1(\U_1))$ and $\rho^2\in\Cont([0,T);\P_1(\U_2))$ to model \eqref{eqn:model} on $\U_1$ and $\U_2$, respectively. Then, $\mu:=\rho^1\otimes\rho^2$ is the unique weak solution to \eqref{eqn:model} among curves in $\Cont([0,T);\P_1(\U))$.
\end{prop}
\begin{proof}
	For all $(x,y)=:z\in\supp(\mu_0) = \supp(\rho^1_0\otimes\rho^2_0)$ and $t\in[0,T)$, set
	\bes
		 \Psi_{\V[\mu]}(z,t) = (\Psi_{\V[\rho^1]}(x,t),\Psi_{\V[\rho^2]}(y,t)) =: (\Psi_{\V[\rho^1]}^t,\Psi_{\V[\rho^2]}^t)(x,y),
	\ees
	where $\Psi_{\V[\rho^1]}$ and $\Psi_{\V[\rho^2]}$ are the unique flow maps generated by $\V[\rho^1]$ and $\V[\rho^2]$ and defined on the time interval $[0,T)$, and compute
	\begin{align*}
		\frac{\d}{\d t} \Psi_{\V[\mu]}(z,t) &= \left( \frac{\d}{\d t} \Psi_{\V[\rho^1]}(x,t), \frac{\d}{\d t} \Psi_{\V[\rho^2]}(y,t) \right)\\
		&= (\V[\rho^1](x,t),\V[\rho^2](y,t)).
	\end{align*}
Further, by using the specific form of $K$ in \eqref{eqn:K-quad}, along with \eqref{eqn:logM}, we find:
        \begin{align*}
		(\V[\rho^1](x,t),\V[\rho^2](y,t)) &= \left( \int_{M_1} \log_x \bar x \d\rho^1_t(\bar x), \int_{M_2} \log_y \bar y \d \rho^2_t(\bar y) \right)\\
		&= \int_{M_1}\int_{M_2} (\log_x \bar x,\log_y \bar y) \d \rho^2_t(\bar y) \d\rho^1_t(\bar x) \\ 
		&= \int_\M \log_{z} \bar z \d \mu_t(\bar z) \\
		&= \V[\mu](z,t).
	\end{align*}
	Hence $\Psi_{\V[\mu]}$ is the unique flow map generated by $v[\mu]$ and defined on the time interval $[0,T)$. Recalling now that, for all $t\in[0,T)$,
\bes
	\rho^1_t = \Psi_{\V[\rho^1]}^t \# \rho^1_0, \qquad \rho^2_t = \Psi_{\V[\rho^2]}^t \# \rho^2_0,
\ees
we have
\bes
	\Psi_{\V[\mu]}^t\#\mu_0 = (\Psi_{\V[\rho^1]}^t,\Psi_{\V[\rho^2]}^t)\# (\rho^1_0\otimes\rho^2_0) = (\Psi_{\V[\rho^1]}^t \# \rho^1_0) \otimes (\Psi_{\V[\rho^2]}^t \# \rho^2_0) = \rho^1_t\otimes \rho^2_t = \mu_t,
\ees
which ends the proof.
\end{proof}
This result gives us the continuum consensus on product manifolds:

\begin{prop}[Asymptotic consensus in continuum product model]
\label{prop:cons-prod}
	Let $K$ be as in \eqref{eqn:K-quad}. Suppose that there exist unique global solutions $\rho^1\in\Cont([0,\infty);\P_1(\U_1))$ and $\rho^2\in\Cont([0,\infty);\P_1(\U_2))$ to model \eqref{eqn:model} on $\U_1$ and $\U_2$, respectively, and suppose they reach asymptotic consensus, that is, there exist $p\in\U_1$ and $q\in\U_2$ such that $W_1(\rho_t^1,\delta_p) \to 0$ and $W_1(\rho_t^2,\delta_q)\to0$ as $t\to\infty$. Assume moreover that there exist $\bar t>0$ and compact sets $Q_1\subset \U_1$ and $Q_2\subset \U_2$ such that $\supp(\rho^1_t) \subset Q_1$ and $\supp(\rho^2_t) \subset Q_2$ for all $t\in[\bar t,\infty)$. Then, the unique global weak solution $\mu$ to \eqref{eqn:model} on $\U$ from Proposition \ref{prop:well-posedness-prod} satisfies $W_1(\mu_t,\delta_{(p,q)})\to0$ as $t\to\infty$.
\end{prop}
\begin{proof}
	Since $\mu = \rho^1\otimes\rho^2$, we get that $(\mu_t)_{t\geq0}$ converges narrowly to $\delta_p\otimes\delta_q = \delta_{(p,q)}$. Furthermore, the compactness of $Q_1\times Q_2$ ensures that in fact $W_1(\mu_t,\delta_{(p,q)})\to0$ as $t\to\infty$.
\end{proof}

\begin{rem}
	In Proposition \ref{prop:cons-prod}, the assumption of support compactness beyond a certain time allows us to infer consensus on the product manifold in the $W_1$ topology from the $W_1$ consensus on each component. Indeed, it ensures the convergence of the first moment of the product measure. 
	
	We observe that this assumption can be relaxed if consensus in each component is regarded in the $W_2$ topology instead, that is, in the topology given by the quadratic Wasserstein distance, as illustrated by the following computation. Given $z_0=(x_0,y_0)\in\U$, the second moment of $\mu_t = \rho^1_t\otimes\rho^2_t$ with respect to $z_0$ for all $t\in[0,\infty)$, satisfies
	\begin{align*}
		\int_\U d(z_0,z)^2 \d\mu_t(z) &= \int_{\U_1} \int_{\U_2} d((x_0,y_0),(x,y))^2 \d\rho^2_t(y)\d\rho^1_t(x)\\
		&= \int_{\U_1} \int_{\U_2} (d_1(x_0,x)^2 + d_2(y_0,y)^2) \d\rho^2_t(y)\d\rho^1_t(x)\\
		&= \int_{\U_1} d_1(x_0,x)^2 \d\rho^1_t(x) + \int_{\U_2} d_2(y_0,y)^2 \d\rho^2_t(y).
	\end{align*}
Then, provided $W_2(\rho_t^1,\delta_p) \to 0$ and $W_2(\rho_t^2,\delta_q)\to0$ as $t\to\infty$, we get
\begin{align*}
	\int_\U d(z_0,z)^2 \d \mu_t(z) \to & \int_{\U_1} d_1(x_0,x)^2 \d\delta_p(x) + \int_{\U_2} d_2(y_0,y)^2 \d\delta_q(y) \\ 
	& = \int_\U d(z_0,z)^2 \d \delta_{(p,q)}(z), \qquad \mbox{as $t\to\infty$},
\end{align*}
and conclude that $W_2(\mu_t,\delta_{(p,q)})\to0$ as $t\to\infty$. 

We further note from the computation above that as an alternative to changing the topology of consensus, the compactness assumption in Proposition \ref{prop:cons-prod} can also be relaxed by choosing $d = d_1 + d_2$ as product distance instead of \eqref{eq:prod-dist}.
\end{rem}


To be able to consider a wider range of product manifolds (see Remark \ref{rem:consensus-prod}), we consider now the case when $\M$ is the real line $\R$, equipped with the canonical topology. The well-posedness and asymptotic consensus on $\R$ with a quadratic attractive potential is well-known in the literature; well-posedness can be obtained by following the same ideas as in the proof of Proposition \ref{prop:inv-cont}, while we present the consensus below for completeness.
\begin{lem}\label{lem:consR}
	Consider model \eqref{eqn:model} on $\R$ with quadratic interaction potential, $K(x,y)= \frac{1}{2}|x -y|^2$, and a compactly supported initial measure $\rho_0$. Then, the unique global weak solution $\rho$ starting from $\rho_0$ reaches asymptotic consensus.
\end{lem}
\begin{proof}
For all $t\in[0,\infty)$, denote by $x_1(t)$ and $x_2(t)$ the left- and right-end points of the support of $\rho_t$, respectively; we know $x_1(t)$ and $x_2(t)$ are finite by the same arguments as given in the proof of Proposition \ref{prop:inv-cont}. Then, using that $\nabla K_y(x) = x - y$ for all $x,y\in\R$, we find, for all $t\in[0,\infty)$:
\begin{align*}
\frac{\der}{\der t}|x_1-x_2|^2(t) &=2 \left( \int_\R (x-x_1(t)) \d \rho_t(x) - \int_\R (x-x_2(t)) \d \rho_t(x)\right) \cdot (x_1(t)-x_2(t)) \\
&=-2|x_1(t)-x_2(t)|^2,
\end{align*}
so that
\[
|x_1(t)-x_2(t)|=|x_1(0)-x_2(0)| \exp\left(- t\right),
\]
which yields the conclusion by taking $t\to0$. 
\end{proof}
Note that the asymptotic consensus in Lemma \ref{lem:consR} is unconditional, in the sense that {\em any} initial configuration with compact support evolves into a consensus state. 
\begin{rem}
\label{rem:consensus-prod}
Using Proposition \ref{prop:cons-prod} one can infer results on asymptotic convergence to a consensus state for a variety of product manifolds, including those illustrated in Example \ref{ex:prod-manifolds}. Indeed, consensus on circle $\mathbb{S}^1$ (or $\mathbb{S}^{\dim}$ in general) is shown in Theorem \ref{thm:consensus-cont}, while consensus on $\R$ is shown in Lemma \ref{lem:consR}.
\end{rem}

The results above can be applied in the discrete setting. In brief, for $\N$ particles $z_i \in \U$, $i\in\{1,\dots,\N\}$, the discrete model with quadratic potential reads:
\be\label{eqn:model-discr-prod}
	\begin{cases} z_i'(t) = \displaystyle \frac1\N  \sum_{j=1}^{\N} \log_{z_i(t)} z_j(t),\\ z_i(0) = z_i^0. \end{cases}
\ee
By \eqref{eqn:logM}, for all $i\in\{1,\dots,\N\}$ the dynamics of $z_i=(x_i,y_i)$ starting from $z_i^0 = (x_i^0,y_i^0)\in\U$ separate into dynamics of $x_i$ on $\U_1$ and of $y_i$ on $\U_2$, respectively, i.e., \eqref{eqn:model-discr-prod} is equivalent to 
\begin{equation}
\label{eqn:discr-decoupled}
	\begin{cases} x_i'(t) = \displaystyle \frac1\N  \sum_{j=1}^{\N} \log_{x_i(t)} x_j(t),\\ x_i(0) = x_i^0, \end{cases} \quad \begin{cases} y_i'(t) = \displaystyle \frac1\N  \sum_{j=1}^{\N} \log_{y_i(t)} y_j(t),\\ y_i(0) = y_i^0. \end{cases}
\end{equation}
Each of the decoupled systems corresponds to the discrete model with quadratic potential on $M_1$ and $M_2$, respectively. Consequently, we directly have the following theorem ensuring that separate consensuses on $\U_1$ and $\U_2$ imply consensus on $\U$:
\begin{prop}[Asymptotic consensus in discrete product model]\label{prop:cons-prod-d}
	Let $K$ be given by \eqref{eqn:K-quad}, and consider the discrete systems in \eqref{eqn:discr-decoupled}. Suppose that these systems have unique global solutions $(x_i)_{i=1}^n$ and $(y_i)_{i=1}^n$ in $\U_1$ and $\U_2$, respectively, which satisfy, for every $i,j\in\{1,\dots,\N\}$,
	\bes
		d_1(x_i(t),x_j(t)) \to 0 \quad \mbox{and} \quad d_2(y_i(t),y_j(t)) \to 0, \qquad \mbox{as $t\to\infty$}.
	\ees
	Then, the unique global solution $(z_i)_{i=1}^\N := ((x_i,y_i))_{i=1}^\N$ to \eqref{eqn:model-discr-prod}, given by Proposition \ref{prop:cons-prod}, verifies, for every $i,j\in\{1,\dots,\N\}$,
	\bes
		d(z_i(t),z_j(t)) \to 0, \qquad \mbox{as $t\to\infty$}.
	\ees
\end{prop}
Finally, using Proposition \ref{prop:cons-prod-d} together with Theorem \ref{thm:consensus-disc} and the discrete version of Lemma \ref{lem:consR}, one can establish asymptotic consensus in the discrete model on product manifolds such as those in Example \ref{ex:prod-manifolds}. 


\begin{appendix}
\section{Appendix}

\subsection{Flows on manifolds}
\label{subsect:A-wp}
We summarize here some standard concepts and results on flow maps generated by vector fields on a smooth, complete and connected $\dim$-dimensional Riemannian manifold $\M$ with intrinsic distance $d$. As in the main body of the paper, $T\in(0,\infty]$ denotes a generic final time and $\U$ a generic open subset of $\M$.
\smallskip

{\em Well-posedness of flow maps.}  Local well-posedness of the flow map equation \eqref{eq:characteristics-general} can de established in local charts using standard ODE theory. To this end, we introduce here the notion of Lipschitz continuity and boundedness of a vector field on $\U$.
\begin{defn}[Lipschitz continuity and boundedness on charts]\label{defn:Lip}
	Let $X$ be a vector field on $\U$. We say that $X$ is \emph{locally Lipschitz continuous on charts} if for every chart $(U,\varphi)$ of $\M$ and compact set $Q\subset U \cap \U$, there exists $L_{\varphi,Q}>0$ such that
	\be\label{eq:Lip-vf}
		\norm{\varphi_*X(x) - \varphi_*X(y)}_{\R^\dim} \leq L_{\varphi,Q} \norm{\varphi(x)-\varphi(y)}_{\R^\dim}, \qquad \mbox{for all $x,y\in Q$};
	\ee
	we denote by $\norm{X}_{\mathrm{Lip}(\varphi,Q)}$ the smallest such constant. We say that $X$ is \emph{locally bounded on charts} if for every chart $(U,\varphi)$ of $\M$ and compact set $Q\subset U\cap \U$, there exists $C_{\varphi,Q}>0$ such that
	\bes
		\norm{\varphi_*X(x)}_{\R^\dim} \leq C_{\varphi,Q}, \qquad \mbox{for all $x\in Q$};
	\ees
	we denote by $\norm{X}_{L^\infty(\varphi,Q)}$ the smallest such constant.
\end{defn}

In the above definition, $\varphi_*$ stands for the push-forward of $\varphi$ in the differential geometric sense. Recall the definition: given $\M_1$ and $\M_2$ two differentiable manifolds, a differentiable function $f\: \M_1 \to \M_2$, a point $x\in\M_1$ and a tangent vector $v\in T_x\M_1$, we call $f_*v := \der f(x)(v) \in T_{f(x)}M_2$ the \emph{push-forward} of $v$ through $f$. In particular, in Definition \ref{defn:Lip} we have $\varphi_*X(x) \in T_{\varphi(x)}\R^\dim \simeq \R^\dim$ and $\varphi_*X \circ \varphi^{-1}$ is a vector field on $\varphi(U)\subset\R^\dim$. 

Given a chart $(U,\varphi)$ of $\M$ containing a point $x\in\M$, one has the basis $\left\{\frac{\p}{\p\varphi^1}(x),\dots,\frac{\p}{\p\varphi^\dim}(x)\right\}$ of $T_x \M$ defined by 
\bes
	\frac{\p}{\p\varphi^i}(x) = \der \varphi^{-1}(\varphi(x))(e_i), \quad \mbox{or} \quad e_i = \der\varphi(x)\left(\frac{\p}{\p\varphi^i}(x)\right) = \varphi_* \frac{\p}{\p\varphi^i}(x),
\ees
where, for all $i\in\{1,\dots,\dim\}$, $e_i$ is the $i$th vector of the canonical basis of $\R^\dim$. Also, for 
\bes
	v = v^1 \frac{\p}{\p\varphi^1}(x) + \cdots + v^\dim \frac{\p}{\p\varphi^\dim}(x) \in T_x\M,
\ees
where $(v^1,\dots,v^\dim)\in\R^\dim$, by linearity we get
\begin{align*}
	\varphi_*v = \der\varphi(x)(v) = v^1 e_1 + \cdots + v^\dim e_\dim.
\end{align*}

\begin{rem}\label{rem:Lip} It is easy to check that local Lipschitz continuity on charts implies local boundedness on charts. One can also see that when $\M=\R^\dim$ equipped with the only chart $(\R^\dim,\mathrm{id})$, in Definition \ref{defn:Lip} we recover the classical Euclidean notion of a locally Lipschitz continuous vector field. 

We also point out that this notion of Lipschitz continuity is consistent with the standard notion of differentiability of a vector field from differential geometry, which says that a vector field $X$ on $\U$ is differentiable at $x\in\U$ if for every chart $(U,\varphi)$ of $\M$ with $x\in U$ the map $\xi \mapsto \varphi_*X(\varphi^{-1}(\xi))$ is differentiable at $\varphi(x)$ in the standard Euclidean sense. Indeed, note that \eqref{eq:Lip-vf} can be equivalently reformulated as: for every chart $(U,\varphi)$ of $\M$ and compact set $R\subset \varphi(U\cap\U)$, there exists $L_{\varphi,R}>0$ such that 
\bes
	\norm{\varphi_*X(\varphi^{-1}(\xi)) - \varphi_*X(\varphi^{-1}(\eta))}_{\R^\dim} \leq L_{\varphi,R} \norm{\xi-\eta}_{\R^\dim}, \qquad \mbox{for all $\xi,\eta\in R$}.
\ees
\end{rem}

We now state and prove the Cauchy--Lipschitz theorem on $\U$.
\begin{thm}[Cauchy--Lipschitz]\label{thm:Cauchy-Lip}
	Let $\a\in(0,\infty]$ and let $X$ be a time-dependent vector field on $\U\times [0,\a)$. Suppose that the vector fields in $\{X_t\}_{t\in [0,T)}$ are locally Lipschitz continuous on charts and satisfy, for any chart $(U,\varphi)$ of $\M$ and compact sets $Q\subset U\cap \U$ and $S\subset [0,\a)$,
	\be\label{eq:bdd-Lip}
		\int_S \left( \norm{X_t}_{L^\infty(\varphi,Q)} + \norm{X_t}_{\mathrm{Lip}(\varphi,Q)} \right) \d t < \infty.
	\ee
Then, for every compact subset $\Sigma$ of $\U$, there exists a unique maximal flow map generated by $(X,\Sigma)$.
\end{thm}
\begin{proof}
	Take $x\in \U$ and choose a chart $(U,\varphi)$ of $\M$ with $x\in U$. Consider the initial-value problem
	\be\label{eq:characteristics-Rn}
		\begin{cases} \alpha'(t) = \Xi(\alpha(t),t),\\ \alpha(0) = \varphi(x),
	\end{cases}
	\ee
	where we define $\Xi\: \varphi(U\cap \U) \times [0,\a) \to \R^\dim$ by
	\bes
		\Xi(\xi,t) = \varphi_*(X_t \circ \varphi^{-1}(\xi)), \qquad \mbox{for all $(\xi,t) \in \varphi(U\cap \U) \times [0,\a)$}.
	\ees
	Take $R\subset \varphi(U\cap\U)$ and $S\subset [0,\a)$ compact, so that in particular $Q:=\varphi^{-1}(R)\subset U\cap\U$ is compact. For all $\xi,\eta\in R$ and $t\in S$, our Lipschitz-continuity assumption on $X_t$ yields
	\bes
		\norm{\Xi(\xi,t) - \Xi(\eta,t)}_{\R^\dim} = \norm{\varphi_*(X_t \circ \varphi^{-1}(\xi)) - \varphi_*(X_t \circ \varphi^{-1}(\eta))}_{\R^\dim} \leq \norm{X_t}_{\mathrm{Lip}(\varphi,Q)} \norm{\xi-\eta}_{\R^\dim}.
	\ees
	Also, for all $\xi\in R$ it holds that
	\bes
		\norm{\Xi(\xi,t)}_{\R^\dim} = \norm{\varphi_*(X_t \circ \varphi^{-1}(\xi))}_{\R^\dim} \leq \norm{X_t}_{L^\infty(\varphi,Q)}.
	\ees
	Therefore, by \eqref{eq:bdd-Lip}we get that $\Xi$ satisfies
	\bes
		\int_S \left( \norm{\Xi(\cdot,t)}_{L^\infty(R)} + \norm{\Xi(\cdot,t)}_{\mathrm{Lip}(R)} \right) \d t \leq \int_S \left( \norm{X_t}_{L^\infty(\varphi,Q)} + \norm{X_t}_{\mathrm{Lip}(\varphi,Q)} \right) \d t < \infty.
	\ees
	By arbitrariness of the compact sets $R\subset \varphi(U\cap\U)$ and $S\subset [0,\a)$ and by the classical Cauchy--Lipschitz theorem on $\R^\dim$, this yields the existence of a unique maximal solution $\alpha_x$ to \eqref{eq:characteristics-Rn} defined on some time interval $[0,\tau_x)$, with $\tau_x\leq a$, and with values in $\varphi(U\cap\U)$. By defining $\Psi_x = \varphi^{-1} \circ \alpha_x$, we see that $\alpha_x$ satisfies \eqref{eq:characteristics-Rn} if and only if
	\be\label{eq:characteristics-Rn-chart}
		\begin{cases} \varphi_*\Psi_x'(t) = (\varphi \circ \Psi_x)'(t) = \varphi_*X_t(\Psi_x(t)) & \mbox{for all $t\in [0,\tau_x)$},\\ \varphi(\Psi_x(0)) = \varphi(x).
	\end{cases}
	\ee
	By the bijectivity of $\varphi$, we get that $\Psi_x$ is thus the unique maximal solution to the characteristic equation \eqref{eq:characteristics-general} starting at $x$. 
	
	Let now $\Sigma$ be a compact subset of $\U$. We are left with showing that $\tau:=\inf_{x\in\Sigma}(\tau_x)>0$. By classical Euclidean Lipschitz theory, we deduce that for all $x\in \U$ there exists $\delta_x>0$ such that $\bar\tau_x:=\inf_{y\in B_{\delta_x}(x)} \tau_y > 0$. Since $\Sigma$ is compact we know it can be covered by a finite subfamily of $\{B_{\delta_x}(x)\}_{x\in\Sigma}$, which we index by $\{x_1,\dots,x_n\}$ for some $n\in\mathbb{N}$. We thus get $\tau = \min_{i\in\{1,\dots,n\}} \bar\tau_{x_i} > 0$. The map $\Psi$ defined by $\Psi(x,t) = \Psi_x(t)$ for all $x\in \Sigma$ and $t\in [0,\tau)$ is then the unique maximal flow map generated by $(X,\Sigma)$.
\end{proof}

From Theorem \ref{thm:Cauchy-Lip} we recover the classical Cauchy--Lipschitz theorem for flow maps in Euclidean space when $\M=\R^\dim$. For completeness, we mention that another important result for flow maps on manifolds is when $\U=\M$ itself is compact. The Escape Lemma \cite[Chapter~9]{Lee2013} states that if an integral curve of a Lipschitz continuous vector field on a manifold is not global (i.e., not defined for all $t \in \R$), then the image of that curve cannot lie in any compact subset of the manifold. Consequently, Lipschitz continuous vector fields on compact manifolds defined at all times (i.e., $a=\infty$ above) generate global flows. Another consequence of the Escape Lemma is the following global version of the Cauchy--Lipschitz theorem, in which we assume that $\U$ is geodesically convex. We recall that then, for all $x,y\in \U$, we have
\be\label{eq:disc-geo-convex}
	 \grad_\M d_y^2(x) = -2\log_x(y) \quad \mbox{and} \quad \norm{\log_x(y)}_x = d(x,y),
\ee
where $d_y(x)$ stands for $d(x,y)$, and $\log$ stands for the Riemannian logarithm on $\M$. We shall also use the notation
\bes
	D_r(p) = \{x\in\M \st d(x,p) < r\}, \quad \mbox{for any $p\in\M$ and $r>0$},
\ees
for the open disk in $\M$ of centre $p$ and radius $r$.

\begin{thm}[Global Cauchy--Lipschitz]\label{thm:global-Cauchy-Lip}
	Suppose that $\U$ is geodesically convex. Under the same hypotheses as those of Theorem \ref{thm:Cauchy-Lip}, where $\Sigma$ is a compact subset of $\U$, suppose moreover that $a=\infty$ and there exist $p\in\U$, $r>0$ and $R>r$ so that $D_R(p)\subset \U$, $\Sigma\subset \overline{D_r(p)}$ and
	\be\label{eq:attractive-general}
		\ap{-\log_x p,X(x,t)}_x \leq 0 \quad \mbox{for all $x\in D_R(p)\setminus D_r(p)$ and $t\in[0,\infty)$}.
	\ee
	Then, there exists a unique flow map $\Psi$ generated by $(X,\Sigma)$ defined on $\Sigma\times [0,\infty)$; furthermore, $\Psi(x,t)\in \overline{D_r(p)}$ for all $(x,t)\in\Sigma\times[0,\infty)$.
\end{thm}
\begin{proof}
	By Theorem \ref{thm:Cauchy-Lip} we know there exists a maximal flow map $\Psi$ generated by $(X,\Sigma)$ defined on $\Sigma\times [0,\tau)$ for some $\tau \in (0,\infty]$. Write $\Psi^t(x)=\Psi(x,t)$ for all $(x,t)\in\Sigma\times[0,\tau)$. Let $x\in \Sigma$ and suppose, by contradiction, that there exists $\tau^*\in(0,\tau)$ such that $\Psi^{\tau^*}(x)\in D_R(p)\setminus \overline{D_r(p)}$. Then, by time continuity of the flow map, we know there exists $\bar\tau\in(0,\tau^*)$ such that $\Psi^{\bar\tau}(x) \in \p D_r(p)$ and $\Psi^{t}(x) \in D_R(p)\setminus D_r(p)$ for all $t\in[\bar\tau,\tau^*]$. Thus, for all $t\in[\bar \tau,\tau^*]$, we have
	\begin{align*}
		\frac{\d}{\d t} d(\Psi^{t}(x),p)^2 &= \ap{\nabla_\M d_p^2(\Psi^t(x)),X(\Psi^t(x),t)}_{\Psi^t(x)} \\
		&= 2 \ap{-\log_{\Psi^t(x)}(p),X(\Psi^t(x),t)}_{\Psi^t(x)} \leq 0,
	\end{align*}
	and by integrating the above between $\bar\tau$ and $\tau^*$ we get
	\bes
		r < d(\Psi^{\tau^*}(x),p) \leq d(\Psi^{\bar \tau}(x),p) = r,
	\ees
	which is absurd. We must therefore have $\Psi^t(x)\in \overline{D_r(p)}$ for all $t\in[0,\tau)$. By the Escape Lemma, this implies that $\tau=\infty$, which ends the proof.
\end{proof}

\subsection{The case of the interaction velocity field}
\label{subsect:A-cont}

In this subsection we show that for a fixed curve $\rho \in \Cont([0,T),\P(\U))$, the velocity field $\V[\rho]$ associated to the interaction equation (see equation \eqref{eqn:v-field}) under Hypothesis \ref{hyp:K} satisfies the assumptions of Theorem \ref{thm:Cauchy-Lip}, and hence it generates a local flow map. This, in particular, justifies the definition of the map $\Gamma$ used in Theorems \ref{thm:well-posedness} and \ref{thm:well-posedness-cyl}. We also show that $\V[\rho]$ satisfies the Theorem \ref{thm:global-Cauchy-Lip} whenever $K$ is purely attractive, that is, $g'\geq0$ in Hypothesis \ref{hyp:K}.

For a curve $\rho  \in \Cont([0,T),\P(\U))$ we recall:
\bes
	\V[\rho](x,t) = - \int_\U \grad_\M K_y(x) \d \rho_t(y), \qquad \mbox{$x\in \U$, $t\in [0,T)$},
\ees
where $K\:\M\times\M\to \R$ is the interaction potential and $K_y$ stands for $x\mapsto K(x,y)$. To ensure that $\V[\rho]$ is pointwise well-defined we can restrict to curves $\rho\in \Cont([0,T),\P_1(\U))$ and assume that there exist measurable functions $\alpha,\beta\: \M \to [0,\infty)$ such that
\bes
	\norm{\grad_\M K_y(x)}_x \leq \alpha(x) + \beta(x) d(x,y), \qquad \mbox{for all $x,y\in\M$}.
\ees

Otherwise, one can also restrict to $\rho\in\Cont([0,T);\P_\infty(\U))$ and assume that the vector field $\grad_M K_y$ is locally bounded on charts for all $y\in\U$. In this case, if we further assume that $\grad_M K_y$ is locally Lipschitz continuous on charts for all $y\in\U$, then Theorem \ref{thm:Cauchy-Lip} applies to the vector field $\V[\rho]$ provided the maps $y\mapsto \norm{\grad_\M K_y}_{L^\infty(\varphi,Q)}$ and $y\mapsto \norm{\grad_\M K_y}_{\mathrm{Lip}(\varphi,Q)}$ are locally bounded for any chart $(U,\varphi)$ of $\M$ and compact set $Q\subset U\cap \U$. 

Indeed, let $\rho\in\Cont([0,T);\P_\infty(\U))$, $(U,\varphi)$ be a chart of $\M$, $Q\subset U\cap\U$ be compact and let $Q_t \subset \U$ be a compact set containing $\supp(\rho_t)$ such that $t\mapsto \mathrm{diam}(Q_t)$ is nondecreasing. Then, for all $x,y\in Q$ and $t\in S$, where $S\subset [0,T)$ is compact, we have
\be\label{eq:bdd-int-vel}
	\norm{\varphi_* \V[\rho](x,t)}_{\R^\dim} \leq \int_{Q_t} \norm{\varphi_*\grad_\M K_y(x)}_{\R^\dim} \d \rho_t(y) \leq \sup_{\bar y\in Q_s} \norm{\grad_\M K_{\bar y}}_{L^\infty(\varphi,Q)},
\ee
where $s=\sup(S)$, and
\begin{align}\label{eq:Lip-int-vel}
	\norm{\varphi_*\V[\rho](x,t) - \varphi_*\V[\rho](y,t)}_{\R^\dim} &\leq \int_{Q_t} \norm{\varphi_*\grad_\M K_z(x) - \varphi_*\grad_\M K_z(y))}_{\R^\dim} \d \rho_t(z) \nonumber\\
	&\leq \sup_{\bar z\in Q_s} \norm{\grad_\M K_{\bar z}}_{\text{Lip}(\varphi,Q)} \norm{\varphi(x)-\varphi(y)}_{\R^\dim}.
\end{align}

In practice, however, it may not be easy to check whether $\grad_\M K$ satisfies this uniform Lipschitz condition because of the push-forward with the chart that needs to be computed. In the particular case when the potential satisfies \ref{hyp:K}, as in our paper, the conditions can be checked as follows. 

Suppose in the rest of this subsection that $\U$ is geodesically convex. In particular, this implies that $\U$ can be covered by a single chart, which we shall generically denote by $(\U,\psi)$; such a chart is for instance provided by any normal chart. Furthermore, the relations in \eqref{eq:disc-geo-convex} hold for all $x,y\in\U$. This allows us to treat the nonlocality in the velocity field which takes the form of an integral on $\U$. 
The lemma below ensures that under Assumption \ref{hyp:K} the Lipschitz theory given in Theorem \ref{thm:Cauchy-Lip} applies to the interaction velocity field $\V[\rho]$ as long as $\rho_t$ is compactly supported for all $t$.

\begin{lem}\label{lem:interaction-complete}
	Let $K$ satisfy \ref{hyp:K}, and let $\rho\in \Cont([0,T);\P_\infty(\U))$. Then the velocity fields in $\{\V[\rho](\cdot,t)\}_{t\in I}$ are locally Lipschitz continuous on charts and satisfy \eqref{eq:bdd-Lip}, that is, they satisfy the assumptions of the Cauchy--Lipschitz theorem \ref{thm:Cauchy-Lip}.
\end{lem}
\begin{proof}
	By the above discussion, we only need to show that $\grad_\M K_z$ is locally Lipschitz continuous on charts and the maps $z\mapsto \norm{\grad_\M K_z}_{L^\infty(\varphi,Q)}$ and $z\mapsto \norm{\grad_\M K_z}_{\mathrm{Lip}(\varphi,Q)}$ are locally bounded for any chart $(U,\varphi)$ on $\M$ and compact set $Q\subset U\cap\U$. Since $\U$ can be covered entirely by a single chart $(\U,\psi)$, we can restrict our computations to $(\U,\psi)$. Note furthermore that since $\M$ is smooth we know that the map $(x,y)\mapsto \grad_\M d^2_y(x)$ is smooth on $\U\times\U$.
	
	Let $Q\subset \U$ be compact. For all $x\in Q$ and $y\in \U$ we get
	\bes
		\norm{\psi_*\grad_\M K_y(x)}_{\R^\dim} \leq |g'(d(x,y)^2)| \norm{\psi_* \grad_\M d^2_y(x)}_{\R^\dim} \leq |g'(d(x,y)^2)| \norm{\grad_\M d^2_y}_{L^\infty(\psi,Q)},
	\ees
	so that by the local boundedness of $g'$ and of $y\mapsto \norm{\grad_\M d^2_y}_{L^\infty(\psi,Q)}$ we get that $y\mapsto \norm{\grad_\M K_y}_{L^\infty(\psi,Q)}$ is locally bounded. Furthermore, for all $x,y \in Q$ and $z\in\U$ we have
	\begin{align*}
		& \norm{\psi_*\grad_\M K_z(x) - \psi_*\grad_\M K_z(y)}_{\R^\dim} = \norm{g'(d(x,z)^2) \psi_*\grad_\M d^2_z(x)  -  g'(d(y,z)^2) \psi_*\grad_\M d^2_z(y)}_{\R^\dim}\\
		&\qquad \leq |g'(d(x,z)^2)| \norm{\psi_*\grad_\M d^2_z(x) - \psi_*\grad_\M d^2_z(y)}_{\R^\dim} + \norm{\psi_*\grad_\M d^2_z(y)}_{\R^\dim} |g'(d(x,z)^2) - g'(d(y,z)^2)| \\
		&\qquad \leq  |g'(d(x,z)^2)| \norm{\grad_\M d^2_z}_{\mathrm{Lip}(\psi,Q)} d(x,y) + \norm{\grad_\M d^2_z}_{L^\infty(\psi,Q)} |g'(d(x,z)^2) - g'(d(y,z)^2)|.
	\end{align*}
Hence, by the local Lipschitz continuity of $g'$ (and thus of $r\mapsto g'(r^2)$), and the local boundedness of $g'$, $z\mapsto \norm{\grad_\M d^2_z}_{\mathrm{Lip}(\psi,Q)}$ and $z\mapsto \norm{\grad_\M d^2_z}_{L^\infty(\psi,Q)}$, we conclude that $z\mapsto \norm{\grad_\M K_z}_{\mathrm{Lip}(\psi,Q)}$ is locally bounded, which ends the proof.
\end{proof}

\begin{rem}\label{rem:indep-max-time}The maximal time of existence of the Cauchy--Lipschitz theorem \ref{thm:Cauchy-Lip} for the interaction velocity field $v[\rho]$ does not depend on the curve $\rho$; this is because the $L^\infty$ and Lipschitz bounds in \eqref{eq:bdd-int-vel} and \eqref{eq:Lip-int-vel} do not depend on $\rho$.
\end{rem}

\begin{lem}\label{lem:interaction-complete-global}
	 Let $K$ satisfy \ref{hyp:K} with $g'\geq 0$, and let $\rho\in \Cont([0,\infty);\P_\infty(\U))$. Let furthermore $\Sigma\subset \U$ be compact and such that $\Sigma\subset \overline{D_r(p)} \subset D_R(p) \subset \U$ for some $p\in\U$ and $r,R>0$ with $\overline{D_\delta(p)}$ geodesically convex for all $\delta\in[r,R)$. Then, the pair $(\V[\rho],\Sigma)$ satisfies the assumptions of the global Cauchy--Lipschitz theorem \ref{thm:global-Cauchy-Lip} provided $\supp(\rho_t)\subset \overline{D_r(p)}$ for all $t\in[0,\infty)$.
\end{lem}
\begin{proof}
Thanks to Lemma \ref{lem:interaction-complete}, we are only left with checking that $\V[\rho]$ verifies \eqref{eq:attractive-general}. Suppose that $\supp(\rho_t)\subset \overline{D_r(p)}$ for all $t\in[0,\infty)$ and let $x\in D_R(p)\setminus D_r(p)$. Then, for all $t\in[0,\infty)$ there holds
\begin{align}\label{eq:integral-global-Lipschitz-velocity}
	\ap{-\log_x p,\V[\rho](x,t)}_x &= \ap{\log_x p,\grad_\M K*\rho_t(x)}_x \nonumber \\
	&= \int_{\overline {D_r(p)}} g'(d(x,y)^2) \ap{\log_x p,\grad_\M d_y^2(x)}_x \d \rho_t(y) \nonumber \\
	&= -2\int_{\overline {D_r(p)}} g'(d(x,y)^2) \ap{\log_x p,\log_x y}_x \d \rho_t(y).
\end{align}

Fix now $y\in \overline{D_r(p)}$ and write $\gamma:[0,1]\to\M$ the unique minimizing geodesic connecting $x$ to $y$. For all $t\in[0,1]$, compute
\bes
	\left. \frac{\der}{\der t}\right|_{t = 0} d(\gamma(t),p)^2 = \ap{\grad_\M d_p^2(\gamma(0)),\gamma'(0)}_{\gamma(0)} = -2\ap{\log_x p,\log_x y}_x.
\ees
By smoothness of $t\mapsto d(\gamma(t),p)^2$ (because $\M$ is smooth), we must have $\left. \frac{\der}{\der t}\right|_{t = 0} d(\gamma(t),p)^2 \leq 0$. Indeed, otherwise there would exist $\tau\in(0,1)$ such that $d(\gamma(\tau),p)>d(\gamma(0),p)=d(x,p)$, which would contradict $\gamma([0,1])\subset \overline{D_{d(x,p)}(p)}$ and thus the geodesic convexity of $\overline{D_{d(x,p)}(p)}$.

Coming back to \eqref{eq:integral-global-Lipschitz-velocity}, we get $\ap{-\log_x p,\V[\rho](x,t)}_x\leq 0$ for all $t\in[0,\infty)$, since inside the integral we have $g'(d(x,y)^2)\geq 0$ and $\ap{\log_x p,\log_x y}_x\geq0$ (by the above argument).
\end{proof}



\subsection{Solution to the interaction equation}
\label{subsect:A-sol}

{\bf Proof of Lemma \ref{lem:distrib-1}.} Let $\tf\in \Cont_\mt{c}^\infty(\U\times (0,T))$ and, for all $x\in\supp(\rho_0)$, define $\pf_x \: [0,T) \to \R$ by
\bes
	\pf_x(t) = \tf(\Psi_{\V[\rho]}^t(x),t), \qquad \mbox{for all $t\in[0,T)$}.
\ees
We have that $\pf_x$ is differentiable for all $x\in\supp(\rho_0)$ with, for all $t\in[0,T)$,
\bes
	\pf_x'(t) = \p_t \tf(\Psi_{\V[\rho]}^t(x),t) + \bigl \langle \V[\rho](\Psi_{\V[\rho]}^t(x),t),\grad_\M \tf(\Psi_{\V[\rho]}^t(x),t) \bigr\rangle_{\Psi_{\V[\rho]}^t(x)}=: \Lambda_t \circ \Psi_{\V[\rho]}^t(x).
\ees
Denote by $Q\subset\U$ and $S\subset (0,T)$ two compact sets such that $\supp(\phi)\subset Q\times S$, and by $L_\tf$ the quantity $\max(\sup|\p_t \tf|,\sup\norm{\grad_M \tf})$, where the supremums are taken over $\supp(\tf)\times(0,T)$. Writing $s=\sup(S)$ and using \eqref{eq:bound-distrib}, we have:
\begin{align*}
	\int_0^T \int_{\supp(\rho_0)} |\pf_x'(t)| \d\rho_0(x) \d t &= \int_0^T \int_{\supp(\rho_0)} \bigl |  \Lambda_t(\Psi_{\V[\rho]}^t(x))\bigr |  \d\rho_0(x) \d t = \int_0^T \int_\U | \Lambda_t(x)| \d\rho_t(x) d t \\
	&= \int_S \int_Q | \p_t \tf(x,t) + \left \langle \V[\rho](x), \grad_\M \tf(x,t) \right \rangle_x | \d\rho_t(x) \d t\\
	&\leq L_\tf \left(s + \int_S \int_Q\| \V[\rho] (x,t) \|_x \d\rho_t(x) \d t \right) < \infty.
\end{align*}
Consider now the calculation above in reverse (just the first two lines), without the absolute value in the integrand. By Fubini's theorem and since $\tf(x,0) = \tf(x,s) = 0$ for all $x\in \U$, we get
\begin{align*}
	&\int_S \int_Q \left( \p_t \tf(x,t) + \bigl \langle \V[\rho] (x,t), \grad_\M \tf(x,t) \bigr \rangle \right) \d\rho_t(x) \d t 
	= \int_0^T \int_{\supp(\rho_0)} \pf_x'(t) \d\rho_0(x) \d t  \\
	&\phantom{=}= \int_{\supp(\rho_0)} \int_S \pf_x'(t) \d t \d \rho_0(x) = \int_{\supp(\rho_0)} \left( \tf(\Psi_{\V[\rho]}^s(x),s) - \tf(x,0) \right) \d\rho_0(x) \\
	&\phantom{=} = \int_\U \tf(x,s) \d\rho_s(x) - \int_\U \tf(x,0) \d\rho_0(x) = 0.
\end{align*}
This shows that $\rho$ is a solution in the sense of distributions to equation \eqref{eqn:model}.

Let us finally prove that in fact $\rho \in \Cont([0,T);\P(\U))$. Take $t\in [0,T)$ and a sequence $(t_k)_{k\geq1} \subset [0,T)$ such that $t_k \to t$ as $k\to\infty$. For all $\tf\in \Cont_\mt{b}(\U)$, as $k\to\infty$ we get
\bes
	\int_\U \tf(x) \d\rho_{t_k}(x) = \int_{\supp(\rho_0)} \tf(\Psi_{\V[\rho]}^{t_k}(x)) \d\rho_0(x) \to \int_{\supp(\rho_0)} \tf(\Psi_{\V[\rho]}^t(x)) \d\rho_0(x) = \int_\U \tf(x) \d\rho_t(x),
\ees
where we used the time continuity of the flow map, that is, $\Psi_{\V[\rho]}^{t_k} \to \Psi_{\V[\rho]}^t$ pointwise, since it solves the first order ODE system \eqref{eq:characteristics-general}, and Lebesgue's dominated convergence theorem.

\end{appendix}

\bigskip

{\large \bf Acknowledgements }  R.F. acknowledges support from NSERC Discovery Grant PIN-341834 during this research. H.P. was supported by National Research Foundation of Korea (NRF-2020R1A2C3A01003881) and Basic Science Research Program through the National Research Foundation of Korea (NRF) funded by the Ministry of Education (2019R1I1A1A01059585). The research in this paper has been initiated at an workshop hosted by the American Institute of Mathematics (AIM) in San Jose, CA, USA. R.F. and F.P. would like to acknowledge AIM for facilitating their research collaboration.


\bibliographystyle{abbrv}
\def\url#1{}
\bibliography{lit}

\end{document}